\documentclass[a4paper,11pt, dvipsnames]{amsart} %{article}
\usepackage{fourier} % If mathematics don't display well using Times, then use Fourier.
\usepackage{amssymb}
\usepackage{amsmath}
\usepackage{amsthm}
\usepackage[all]{xy}
\usepackage[english]{babel}
\usepackage{tikz}
\usepackage{tikz-cd}
\usetikzlibrary{matrix}
\usetikzlibrary{calc,intersections}
\usepackage[shortlabels]{enumitem}
\usepackage{array, multirow}

\theoremstyle{plain}
\newtheorem{thm}{Theorem}[section]
\newtheorem{mainthm}{Main Theorem}
\newtheorem{prop}[thm]{Proposition}
\newtheorem{lemma}[thm]{Lemma}
\newtheorem{cor}[thm]{Corollary}
\theoremstyle{definition}
\newtheorem{df}[thm]{Definition}

\newtheorem*{notn*}{Notation}

\newtheorem{remark}[thm]{Remark}

\newtheorem{example}[thm]{Example}

\newcounter{enum_counter}

\DeclareMathOperator{\End}{End}

\DeclareMathOperator{\Tr}{Tr}

\DeclareMathOperator{\ord}{ord}
\DeclareMathOperator{\Jac}{Jac}

\DeclareMathOperator{\type}{type}

\def\Q{\mathbb{Q}}
\def\Z{\mathbb{Z}}

\def\F{\mathbb{F}}

\def\p{\mathfrak{p}}
\def\frf{\mathfrak{f}}

\newcommand{\cC}{\mathcal{C}}
\newcommand{\cF}{\mathcal{F}}

\newcommand{\cI}{\mathcal{I}}

\newcommand{\cO}{\mathcal{O}}

\newcommand{\cR}{\mathcal{R}}

\newcommand{\rad}[1]{\mathrm{rad}(#1)}

\newcommand{\vphi}{\varphi}
\newcommand{\set}[1]{\left\lbrace#1\right\rbrace }

\renewcommand{\bar}{\overline}
\renewcommand{\hat}{\widehat}
\usepackage{hyperref}

\usepackage{color}

\title{Abelian varieties over finite fields and their groups of rational points}

\date{\today}

\author{Stefano Marseglia}
\address{Mathematical Institute, Utrecht University, P.O. Box 80010, 3508 TA, Utrecht, The Netherlands}
\email{s.marseglia@uu.nl}
\thanks{Marseglia is supported by NWO grant VI.Veni.202.107.}

\author{Caleb Springer}
\address{Department of Mathematics, University College London, Gower Street, London WC1H 0AY,
UK\\
and The Heilbronn Insitute for Mathematical Research, Bristol, UK}
\email{c.springer@ucl.ac.uk}
\thanks{ Springer is partially supported by National Science Foundation award {CNS-2001470} and by the Additional Funding Programme for Mathematical Sciences, delivered by EPSRC (EP/V521917/1) and the Heilbronn Institute for Mathematical Research.}


 % Abelian varieties of dimension~$> 1$

\begin{document}

\begin{abstract}
  We study the groups of rational points of abelian varieties defined over a finite field $ \mathbb{F}_q$ whose endomorphism rings are commutative, or, equivalently, whose isogeny classes are determined by squarefree characteristic polynomials. When $\mathrm{End}(A)$ is locally Gorenstein, we show that the group structure of $A(\mathbb{F}_q)$ is determined by $\mathrm{End}(A)$. Moreover, we prove that the same conclusion is attained if $\mathrm{End}(A)$ has local Cohen-Macaulay type at most $ 2$, under the additional assumption that $A$ is ordinary or $q$ is prime. The result in the Gorenstein case is used to characterize squarefree cyclic isogeny classes in terms of conductor ideals. Going in the opposite direction, we characterize squarefree isogeny classes of abelian varieties with $N$ rational points in which every abelian group of order $N$ is realized as a group of rational points. Finally, we study when an abelian variety $A$ over $\mathbb{F}_q$ and its dual $A^\vee$ succeed or fail to satisfy several interrelated properties, namely $A\cong A^\vee$, $A(\mathbb{F}_q)\cong A^\vee(\mathbb{F}_q)$, and $\mathrm{End}(A)=\mathrm{End}(A^\vee)$. In the process, we exhibit a sufficient condition for $A\not\cong A^\vee$ involving the local Cohen-Macaulay type of $\mathrm{End}(A)$. In particular, such an abelian variety $A$ is not a Jacobian, or even principally polarizable.
\end{abstract}

\maketitle
%%%%%
%INTRO
%%%%%

\section{Introduction}
The groups of rational points of abelian varieties defined over a finite field~$\F_q$ have recently received a considerable amount of attention. 
For example, in \cite{HK21}, Howe and Kedlaya showed that every positive integer occurs as the order of the group of rational points of an abelian variety over~$\F_2$.
In \cite{vBCLPS21}, Van~Bommel, Costa, Li, Poonen and Smith prove, among other results,  a version of this statement over arbitrary finite fields~$\F_q$, although only sufficiently large orders are realizable if~$q \geq 7$.

These statement are, in fact, results about isogeny classes.
Indeed, two abelian varieties~$A$ and~$B$ are isogenous over~$\F_q$ if and only if~$\#A(\F_{q^n}) = \#B(\F_{q^n})$ for all~$n\geq1$, or equivalently, if~$h_A(x) = h_B(x)$, where~$h_A, h_B\in \Z[x]$ are the characteristic polynomials of  the Frobenius endomorphisms of~$A$ and~$B$, respectively; see \cite[Thm.~1.(c)]{Tate66}.
Moreover, one has~$\#A(\F_q)=h_A(1)$.

The recent results mentioned above concerning cardinalities were upgraded to statements about finite abelian groups in \cite{MarSpr_pts_PAMS}: for each~$q$ in~$\{2,3,4,5\}$, we showed that every finite abelian group is isomorphic to the group of rational points of some abelian variety over~$\F_q$.
 This upgrade, however, does not work on the level of isogeny classes. 
Indeed, the group structure of~$A(\F_q)$ is not uniquely determined by its isogeny class. This phenomenon is observed even for elliptic curves. In this paper, we seek to understand and describe this extra level of structure.

\subsection*{Notation and conventions}
Before presenting our main results, we set some notation and conventions.
Throughout the paper, all isogenies and morphisms between abelian varieties over~$\F_q$ are defined over the base field~$\F_q$.
In particular, given an abelian variety~$A$  over~$\F_q$ with characteristic polynomial~$h$, we denote its ($\F_q$-rational) endomorphism ring by~$\End(A)$ and its ($\F_q$-)isogeny class by~$\cI_h$.
 We say that~$A$ and~$\cI_h$ are \emph{squarefree} if
$h$ is a squarefree polynomial.
An equivalent definition of squarefree is given by requiring the endomorphism algebra~$\End(A)\otimes_\Z \Q$ to be commutative; see \cite[Thm.~2.(c)]{Tate66}.
See also \cite[Lem.~2.3]{MarSpr_pts_PAMS} for a comparison with other notions of squarefree.

Given a squarefree isogeny class~$\cI_h$ over~$\F_q$, set~$K=\Q[x]/(h)$ and let~$\pi$ be the class of~$x$ in~$K$. We will denote by~$\cO_K$ the maximal order of~$K$.
For every~$A$ in the isogeny class, as in \cite[\S3.1]{Wat69}, we fix an isomorphism~$\End(A)\otimes_\Z \Q \cong K$ which sends the Frobenius endomorphism of~$A$ to~$\pi$.
From now on, we will identify~$\End(A)$ with its image inside~$K$, which is an order.
Under this identification, the Rosati involution of~$A$ acts as the complex conjugation~$x\mapsto \bar{x}$ in~$K$. Note that~$\bar\pi=q/\pi$.
In particular,  if~$\End(A)=S\subset K$ then~$\End(A^\vee)=\bar{S}$, where~$A^\vee$ denotes the dual abelian variety of~$A$.

 \subsection{Groups of rational points and endomorphism rings}
As noted above, given an abelian variety~$A$ over a finite field~$\F_q$, the sequence~$(\#A(\F_{q^n}))_{n\geq 1}$  of point counts is an isogeny invariant, but the group structure of~$A(\F_q)$ is not.
Classifying abelian varieties over finite fields up to isogeny is the same as classifying them according to their endomorphism algebra.
A more precise classification is given by the endomorphism ring, which is an order in the endomorphism algebra.
When~$E$ is an elliptic curve over~$\F_q$, the group structure of~$E(\F_{q^n})$ is uniquely determined for all~$n\geq 1$ by the endomorphism ring~$\End(E)$; see \cite[Thm.~1]{Lenstra96}.

In Main Theorem~\ref{main-thm:main-tools}, we exhibit a similar result for abelian varieties of arbitrary dimension, under certain hypotheses on the endomorphism ring which are automatically satisfied in the case of elliptic curves. 
In particular, we use the notion of the (Cohen-Macaulay) type of an order~$S$ at a prime~$\p$.
This type, denoted~$\type_\p(S)$, is defined as the minimal number of generators of~$(S^t)_\p= S^t\otimes_S S_\p$, where~$S_\p$ is the localization of~$S$ at~$\p$ and~$S^t$ is the trace dual ideal of~$S$.
The order~$S$ is Gorenstein at~$\p$ if~$\type_\p(S) = 1$. 
We say that~$S$ is Gorenstein, if it so at every prime. 
See Section~\ref{sec:frac_ideals} for definitions and details.
Recall that an abelian variety~$A$ over~$\F_q$ is called \emph{ordinary} if the coefficient of~$x^{\dim A}$ in the characteristic polynomial~$h(x)$ of~$A$ is coprime to~$q$.
\begin{mainthm}
  \label{main-thm:main-tools}
  Let~$A$ be an abelian variety in a squarefree isogeny class~$\cI_h$ over~$\F_q$.
  Write~$S=\End(A)$ and fix~$n\geq 1$.
  \begin{enumerate}[(a)]
    \item  
    If~$S$ is Gorenstein at all prime ideals containing~$(1-\pi^n)$, then 
   ~$$
      A(\F_{q^n}) \cong 
      \frac{S}{(1-\pi^n)S}
   ~$$ 
    are isomorphic as~$S$-modules. 
    \label{main-thm:main-tools:Gor}
    \item Assume~$\cI_h$ is ordinary \emph{({\bf Ord})} or~$q = p$ is prime \emph{({\bf CS})}.
    If~$\type_\p(S)\leq 2$
    for every prime~$\p$ of~$S$ above~$(1-\pi^n)$,
    then
    \[ 
      A(\F_{q^n})\cong \frac{S}{(1-\pi^n) S}
    \]
    are isomorphic as~$\Z$-modules. \label{main-thm:main-tools:cmtype2}
  \end{enumerate}
\end{mainthm}
In contrast, in Example~\ref{ex:diff_grps_same_end}, we show that it is possible to have~$A(\F_q)\not\cong B(\F_q)$ for isogenous abelian varieties~$A$ and~$B$ over~$\F_q$, even if~$\End(A) = \End(B)$.
This example, like all the others in this paper, has been computed with the help of Magma \cite{Magma}. 

Part \ref{main-thm:main-tools:Gor}, appearing in the text as Corollary~\ref{cor:Gorenstein}, is proven by generalizing the methods of \cite{Lenstra96} and \cite{Springer21}.
Specifically, we view the group of rational points $A(\F_q)$ of an abelian variety~$A$  over~$\F_q$ as a module over the endomorphism ring $\End(A)$ and use the Gorenstein property to describe the module, as desired.

For Part \ref{main-thm:main-tools:cmtype2},
we use the additional assumptions to consider the abelian variety~$A$ itself to ``be'' a module over~$\End(A)$, as we now explain.
In \cite{Del69}, Deligne constructed an equivalence between the category of ordinary abelian varieties over a finite field~$\F_q$ and the category of free~$\Z$-modules with a ``Frobenius''-like endomorphism.
In \cite{CentelegheStix15}, Centeleghe-Stix extended Deligne's result, using a different functor, to the category of abelian varieties over a prime field~$\F_p$ whose characteristic polynomial does not have real roots. 

Given a squarefree isogeny class~$\cI_h$ over~$\F_q$, we write {\bf Ord} for the condition that~$\cI_h$ is ordinary, and {\bf CS} for the condition that~$q = p$ is prime. Note that, if~$\cI_h$ is squarefree, then the characteristic polynomial~$h$ does not have real roots.
If we restrict the functor of Deligne (resp.~Centeleghe-Stix) to a particular squarefree isogeny class~$\cI_h$ satisfying {\bf Ord} (resp.~{\bf CS}), then the modules in the image of the functor are precisely the fractional~$\Z[\pi,\bar \pi]$-ideals in the endomorphism algebra~$K = \Q[x]/(h)$. 
See Section~\ref{sec:ord-cs} below, or \cite{MarAbVar18} for a detailed account.
In particular, this lends itself to another route for describing groups of rational points in terms of orders and fractional ideals in the endomorphism algebra; see Theorem~\ref{thm:cat_eq}. This description allows us to deduce Main Theorem~\ref{main-thm:main-tools}.\ref{main-thm:main-tools:cmtype2}, written as Proposition~\ref{prop:A_cmtype2_gp} below.

Beyond Main Theorem~\ref{main-thm:main-tools}, the techniques and perspectives introduced in this section continue to be used extensively throughout the paper.
In Section~\ref{sec:frac_ideals}, we recall the necessary background and prove some foundational results regarding orders in \'etale algebras, which we then use in the remainder of the paper.

\subsection{Cyclicity}\
In Section~\ref{sec:coprime_conductor}, we study isogeny classes which are \emph{cyclic}, meaning that every abelian variety in the isogeny class has a cyclic group of points. 
A criterion for cyclicity which only involves the characteristic polynomial was given in \cite[Thm.~2.2]{Giangreco-Maidana19}.
Although this criterion applies \emph{a priori} to all isogeny classes, we prove in Theorem~\ref{thm:general_cyclic_sqfree} that an isogeny class over~$\F_q$ is cyclic if and only if it contains a variety of the form~$A_{\text{sf}}\times A_1$ where~$A_{\text{sf}}$ is squarefree and~$A_1$ has only one rational point. 
Moreover,~$A_1$ must be~$0$-dimensional if~$q \geq 5$ by the Weil bounds.

\newpage
We provide a new criterion for cyclicity, written below as Theorem~\ref{thm:coprime_iff_cyclic}.

\begin{mainthm} 
  \label{mainthm:coprime_iff_cyclic}
  Consider a squarefree isogeny class~$\cI_h$ of abelian varieties over~$\F_q$.
  Let~$\pi$ be the class of~$x$ in the endomorphism algebra~$K = \Q[x]/(h)$.
  The isogeny class~$\cI_h$ is cyclic if and only if~$(1-\pi)\Z[\pi, \overline\pi]$ is coprime to the conductor~$\frf = (\cO_K : \Z[\pi, \overline\pi])$.
\end{mainthm}

Rather than relating the property of cyclicity to the coefficients of the characteristic polynomial as in~\cite{Giangreco-Maidana19}, Main Theorem~\ref{mainthm:coprime_iff_cyclic} relates cyclicity to the algebraic properties of orders in the endomorphism algebra. 
In particular, it shows that the property of cyclicity is equivalent to the local maximality of the order~$\Z[\pi, \overline\pi]$ generated by Frobenius and Verschiebung at all primes over~$(1-\pi)$.
A~key ingredient in our proof is Main Theorem~\ref{main-thm:main-tools}.\ref{main-thm:main-tools:Gor}, applied to abelian varieties~$A$ with maximal endomorphism ring.

\subsection{Richness and non-cyclic groups}\
In contrast to cyclic isogeny classes, we inspect the opposite extreme in Section~\ref{sec:non-cyclic}.
We say that a squarefree isogeny class~$\cI_h$ is \emph{rich} if every abelian group of order~$h(1)$ occurs as the group of rational points of some abelian variety in~$\cI_h$.
In previous work by the authors, it was shown that, for each~$N\geq 1$, there are infinitely many rich squarefree isogeny classes~$\cI_h$ over~$\F_2$ with~$N = h(1)$; see \cite[Thm.~5.3]{MarSpr_pts_PAMS}.
These isogeny classes are built from Kedlaya's infinite sets of simple isogeny classes of abelian varieties over~$\F_2$ with prescribed numbers of points; see \cite[Thm.~1.1]{Ked21}. 
We must use a different technique to find rich isogeny classes in general because there are at most finitely many simple abelian varieties over~$\F_q$ with a prescribed number of points~$N$ when~$q > 2$ by \cite{Kadets21}.

We present a criterion for richness in Main Theorem~\ref{main-thm:rich} which is easy to compute using only the characteristic polynomial. An expanded statement is proved as Theorem~\ref{thm:rich_condition}, and we compare the conditions of cyclicity and richness for abelian varieties of small dimension over small finite fields in Example~\ref{ex:rich_vs_cyclic}.
\begin{mainthm} 
  \label{main-thm:rich}
  Consider a squarefree isogeny class~$\cI_h$ of abelian varieties over~$\F_q$ of dimension~$g$.
  Let~$K = \Q[x]/(h)$ be the endomorphism algebra, and let~$\pi$ be the class of~$x$.
  Write~$N = h(1) = \prod_{j = 1}^s \ell_{j}^{e_j}$ for the number of rational points on each abelian variety in~$\cI_h$.
  The following are equivalent.
  \begin{enumerate}[(a)]
    \item~$\cI_h$ is rich, that is, every abelian group of order~$N$ arises as~$A(\F_q)$ for some~$A\in \cI_h$.
    \item For all~$1\leq i\leq 2g$, we have
   ~$$
      \frac{h^{(i)}(1)}{i!}\cdot \ell_1^{i-e_1}\cdots \ell_s^{i-e_s} \in \Z.
   ~$$
  \end{enumerate} 
\end{mainthm}

To obtain this theorem, we first prove Lemma~\ref{lem:GiangrecoMaidana}, generalizing a lemma of Giangreco-Maidana \cite[Lem.~2.1]{Giangreco-Maidana19} which was originally used to study cyclicity. 
We then deduce Main Theorem~\ref{main-thm:rich} by applying a result of Rybakov \cite[Thm.~1.1]{Rybakov10}.
As a consequence, we also prove that a squarefree isogeny class is rich if and only if its simple factors are rich; see Corollary~\ref{cor:rich_sqfree_simple}.

We conclude Section~\ref{sec:non-cyclic} by proving the existence of some abelian varieties whose groups of rational points have at least two generators. 
In particular, we show in Corollary~\ref{cor:divisible_by_4} that a squarefree isogeny class~$\cI_h$ over~$\F_q$ is non-cylic if~$q$ is odd and~$h(1)$ is divisible by~$4$.
Finally, we prove the existence of ordinary abelian varieties over~$\F_4$ with certain prescribed non-cyclic groups of rational points in Theorem~\ref{thm:MS_improved}, thereby improving \cite[Thm.~3.3]{MarSpr_pts_PAMS}.

\subsection{Duality}\
In Section~\ref{sec:dual}, we turn our attention to the dual~$A^\vee$ of an abelian variety~$A$.
At the AMS MRC, \emph{Explicit Methods in Arithmetic Geometry in Characteristic~$p$}, in June 2019, Bjorn Poonen suggested the problem of finding an abelian variety~$A$ defined over a finite field~$\F_q$ such that~$A(\F_q) \not\cong A^\vee(\F_q)$. In Example~\ref{ex:dual_non_isom_gps}, we find such a variety by using Main Theorem~\ref{main-thm:main-tools}.\ref{main-thm:main-tools:Gor}.
In this example, we observe that~$\End(A)$ is a Gorenstein order and~$\End(A) \neq \End(A^\vee)$.
In Example~\ref{ex:stats_dual_non_isom_gps}, we show that these examples are not rare.

Section~\ref{sec:dual} concludes with a further investigation of the relationships between these properties, along with the properties of being a Jacobian, principally polarizable, or self-dual.
More precisely, when~$A$ is squarefree, consider the following well-known implications, which are recalled below in Theorem~\ref{thm:jac_ppav_implications}.
$$
\xymatrix@R=1em{   
    &  & & A(\F_q)\cong A^\vee(\F_q) \\
    A\cong\Jac(C) \ar@2{->}[r]  & A \text{ has a princ. pol.} \ar@2{->}[r] & A \cong A^\vee\ar@2{->}[ur]\ar@2{->}[dr]& \\\
    &   & & \End(A) = \End(A^\vee) 
}
$$
Examples~\ref{ex:pp_not_jac}, \ref{ex:sd_not_pp} and \ref{ex:same_end_not_sd}, which are likely unsurprising to experts,
illustrate that none of the reverse implications are true.
Additionally, Example~\ref{ex:same_gp_not_same_end} exhibits an abelian variety~$A$ over~$\F_3$ such that~$A(\F_3)\cong A^\vee(\F_3)$ but~$\End(A) \neq \End(A^\vee)$, hence there is no downward implication on the right side of the diagram.
In each case, there are many suitable examples.
On the other hand, it is unknown whether there are examples where~$A(\F_q) \not\cong A^\vee(\F_q)$ and~$\End(A) = \End(A^\vee)$.
Observe that, under the hypotheses of either part of Main Theorem~\ref{main-thm:main-tools},~$\End(A) = \End(A^\vee)$ implies that~$A(\F_q) \cong A^\vee(\F_q)$; see also Proposition~\ref{prop:A_same_grp_Avee}.

Main Theorem~\ref{main-thm:not_self_dual} provides a sufficient condition for~$A\not\cong A^\vee$ which only depends on the properties of the orders in the endomorphism algebra. It is a key ingredient for producing Example~\ref{ex:same_end_not_sd}, and may be of independent interest. It is proved in the text as Proposition \ref{prop:not_self_dual}.
\begin{mainthm}
  \label{main-thm:not_self_dual}
  Let~$A$ be an abelian variety in a squarefree isogeny class~$\cI_h$ over~$\F_q$,
  let~$S$ be an order in~$K = \Q[x]/(h)$ such that~$S=\bar{S}$, and let~$\p$ be a prime of~$S$ satisfying~$\p=\bar{\p}$ and~$\type_\p(S)=2$.
  Assume that~$A$ is ordinary \emph{(}{\bf Ord}\emph{)} or that~$q$ is prime \emph{(}{\bf CS}\emph{)}.
  If~$S\subseteq \End(A)$ and~$S_\p = \End(A)_\p$, then~$A\not\cong A^\vee$.
  In particular, such an~$A$ is not principally polarizable and cannot be a Jacobian.
\end{mainthm}

\subsection{Related literature}
We conclude the introduction by mentioning some additional related results.
There are several papers on the classification of the groups of rational points of elliptic curves; see for example \cite{Tsfasman85}, \cite{Ruck87}, \cite{TVN07} and \cite{Vol88}.
The cases of abelian surfaces and threefolds were studied in \cite{Xing94}, \cite{Xing96}, \cite{Rybakov12}, \cite{Rybakov15}, \cite{Davidetal14}, and \cite{Kotelnikova19}.
Additional results about cyclic isogeny classes can be found in \cite{Giangreco-Maidana20}, \cite{Giangreco-Maidana21} and \cite{BerGM22}.

\subsection*{Acknowledgements}
We thank Jonas Bergstr\"om, Valentijn Karemaker, and Bjorn Poonen for providing comments on a draft version of this paper.

\newpage
%%%%%
%FRACTIONAL IDEALS
%%%%%
\section{Fractional ideals in orders}
\label{sec:frac_ideals}

In this section we recall definitions and properties of orders and their fractional ideals.
These concepts are well-known in the context of number fields, but
we will work in a more general setting.
Additional details and proofs can be found in \cite[Sec.~2]{MarsegliaCMType_arxiv}.

Let~$Z$ be a Dedekind domain with field of fractions~$Q$.
In practice, for the purpose of this paper, it will be enough to consider~$Z=\Z$ and~$Z=\Z_p$.
Let~$K$ be a finite \emph{\'etale algebra} over~$Q$, that is, a finite product of finite separable extensions of~$Q$.
A~$Z$-\emph{lattice}~$L$ in~$K$ is a finitely generated free sub-$Z$-module of~$K$ such that~$L\otimes_Z Q = K$.
Given two lattices~$L_1$ and~$L_2$ in~$K$, we define the \emph{colon} as
\[ (L_1 : L_2) = \set{ x \in K : xL_2 \subseteq L_1}. \]

A~$Z$-\emph{order}~$S$ in~$K$ is a subring of~$K$ which is also a~$Z$-lattice.
Observe that~$K$ is the total ring of quotients of any~$Z$-order~$S$ in~$K$.
When no confusion can arise, we will drop the base ring~$Z$ from the terminology and simply write lattice and order.
When~$S\subseteq S'$ are orders in~$K$, the colon~$(S:S')$ is called the \emph{conductor} of~$S$ in~$S'$. 

A \emph{fractional~$S$-ideal}~$I$ is a finitely generated~$S$-submodule of~$K$ which is also a lattice.
Given a lattice~$L$ in~$K$ we define its \emph{multiplicator ring} as~$(L:L)$.
Observe that~$(L:L)$ is a order in~$K$ and hence~$L$ is a fractional~$(L:L)$-ideal.

Given two lattices~$I$ and~$J$ in~$K$ (resp.~fractional~$S$-ideals) then
the sum~$I+ J$, the intersection~$I\cap J$, the product~$IJ$, and the colon~$(I:J)$ are lattices in~$K$ (resp.~fractional~$S$-ideals).

Let~$S$ be an order in~$K$.
A \emph{prime} of~$S$ is a maximal ideal of~$S$.
We denote by~$S_\p$ the localization of~$S$ at~$\p$, and by~$\hat{S}_\p$ the completion of~$S$ at~$\p$.
For an~$S$-module~$M$ we put~$M_\p=M\otimes_S S_\p$ and~$\hat M_\p=M\otimes_S \hat S_\p$.
Observe that~$\hat{S}_\p$ is a~$\hat{Z}_p$-order, where~$p$ is the contraction of~$\p$ in~$Z$. 
Also, if~$I$ is a fractional~$S$-ideal then~$\hat{I}_\p$ is a fractional~$\hat{S}_\p$-ideal.
We will say that~$I$ is \emph{principal at~$\p$} if~$I_\p$ is a principal~$S_\p$-module, or, equivalently,~$\hat I_\p$ is a principal fractional~$\hat S_\p$-ideal.
A fractional~$S$-ideal~$I$ is called \emph{invertible} if~$I(S:I)=S$ or, equivalently,
if it is principal at~$\p$ for every prime~$\p$ of~$S$. See for example \cite[Lem.~2.12 and~2.17]{MarsegliaCMType_arxiv}.
Given orders~$S\subseteq S'$ in~$K$, we can consider~$S'$ as a fractional~$S$-ideal.
Lemma~\ref{lemma:rel_cond_same_order} below tells us when~$S'$ is principal at a prime~$\p$ of~$S$.
\begin{lemma}
\label{lemma:rel_cond_same_order}
    Let~$S\subseteq S'$ be orders.
    Given a prime~$\p$ of~$S$, the following statements are equivalent:
    \begin{enumerate}[(a)]
    \item~$(S:S') \subseteq \p$.
    \item~${S'}_\p$ is not a principal~$S_\p$-module.
    \item~$S_\p \neq {S'}_\p$.
    \end{enumerate}
\end{lemma}
\begin{proof}
    We first prove that~${S'}_\p$ is a principal~$S_\p$-module if and only if~$S_\p={S'}_\p$.
    One implication is trivial. For the other, assume that~${S'}_\p=\alpha S_\p$.
    Hence~$\alpha \in {S'}_\p^\times$, which shows that
    \[ S_\p = \frac 1\alpha {S'}_\p = {S'}_\p.\]
    To conclude, observe that~$S_\p \subsetneq {S'}_\p$ if and only if~$(S_\p:{S'}_\p) \subsetneq S_\p$, which occurs if and only if~$(S:S') \subseteq \p$.
\end{proof}

Later in the paper we will study groups of rational points of abelian varieties over finite fields by describing them in terms of quotient of fractional ideals.
In turn, we describe such quotients locally.
\begin{lemma}\label{lem:dir_sum_loc}
  Let~$J\subseteq I$ be fractional ideals over an order~$S$.
  Then we have an isomorphism of~$S$-modules
  \[ \frac{I}{J} \simeq \bigoplus_\p \left( \frac{I}{J} \right)_\p, \]
  where the direct sum is over the finitely many primes for which~$I_\p\neq J_\p$.
\end{lemma}
\begin{proof}
  The quotient~$I/J$ is a finitely generated torsion~$Z$-module; see for example~\cite[Sec.~1.7]{cohenadv00}.
  Hence,~$I/J$ is an Artinian and Noetherian~$S$-module, so the result follows from \cite[Thm.~2.13.(b)]{eis95}.
\end{proof}

\begin{prop}\label{prop:converse_O_K}
  Let~$S\subsetneq S'$ be orders.
  If~$r\in S$ is not a zero-divisor
  and~$rS$ is not coprime to the conductor~$\frf = (S : S')$, 
  then the quotient~$M = S'/rS'$ is not a cyclic~$S$-module.
\end{prop}
\begin{proof}
    By assumption there exists a maximal ideal~$\p$ of~$S$ such that
    \[ rS + \frf \subseteq \p, \]
    which implies
    \[ rS' + \frf \subseteq \p S'. \]
    Since~$\p$ is above the conductor~$\frf$, by Lemma \ref{lemma:rel_cond_same_order} we have that~${S'}_\p$ is not a principal~$S_\p$-module, or equivalently the~$S/\p$-vector space~$S'/\p S'$ has dimension strictly bigger than~$1$.
    Observe that
    \[ \frac{M}{\p M} \cong \dfrac{S'/rS'}{\p S'/rS'} \cong \frac{S'}{\p S'}. \]
    Hence the~$S_\p$-module~$M_\p$ is not cyclic.
    By Lemma~\ref{lem:dir_sum_loc}, we conclude that~$M$ is not a cyclic~$S$-module.
\end{proof}

The isomorphism class of a quotient of fractional ideals can be deduced from local information at finitely many primes, as explained in the following lemma.
\begin{lemma}\label{lemma:loc_glob_fin_mod}
  Let~$I$ and~$J$ be fractional ideals over an order~$S$, and let~$r\in S$ be a non-zero divisor.
  Let~$\mathcal{S}$ be the set of primes of~$S$ containing~$r$.
  Assume that we have an~$S_\p$-linear isomorphism~$\vphi_\p:I_\p \overset{\sim}{\to} J_\p$ for every~$\p\in\mathcal{S}$.
  Then there is an~$S$-linear isomorphism
  \[ \psi:\frac{I}{rI}\overset{\sim}{\longrightarrow}\frac{J}{rJ}, \]
  such that~$\psi \otimes S_\p = \vphi_\p\otimes (S/r S)$ for every prime~$\p\in \mathcal{S}$.
\end{lemma}
\begin{proof}
  Set~$M=I/rI$ and~$N=J/rJ$.
  By assumption, we have isomorphisms
  \[ \vphi_\p \otimes \left(\frac{S}{r S}\right) : M_\p \overset{\sim}{\longrightarrow} N_\p, \]
  for each~$\p\in\mathcal{S}$.
  We now claim that, given a prime~$\p$ of~$S$, we have~$M_\p\neq 0$ if and only if~$r \in \p$.
  If~$r\not\in\p$ then~$(S/r S)_\p=0$ and hence~$M_\p = (I\otimes_S (S/r S))=0$ as well.
  Conversely, if~$r \in \p$ then~$r I\subseteq \p I$ and hence we obtain a surjective map
 ~$M_\p \to I/\p I$,
  which shows that~$M_\p\neq 0$, completing the proof of the claim.
  The same is true for the~$S$-module~$N$.

  By Lemma~\ref{lem:dir_sum_loc}, we have
  \[ M \cong \bigoplus_{\p \in \mathcal{S}} M_\p \text{ and } N \cong  \bigoplus_{\p \in \mathcal{S}} N_\p. \]
  We conclude by setting
  \[ \psi = \bigoplus_{\p \in \mathcal{S}} \vphi_\p \otimes \left(\frac{S}{r S}\right). \]
\end{proof}

The \'etale algebra~$K$ comes equipped with a \emph{trace} map
\[ \Tr_{K/Q}: K \to Q \]
that associates to every element~$x\in K$ the trace of the matrix representing the multiplication-by-$x$ map with respect to any~$Q$-basis of~$K$.
The existence of such a non-degenerate trace implies that the integral closure~$\cO_K$ of~$Z$ in~$K$ is an order, called the \emph{maximal order}, since every other order is contained in~$\cO_K$.
Recall that~$\cO_K$ is characterized by the fact that every localization is a principal ideal ring.

The following proposition refines Lemma~\ref{lemma:loc_glob_fin_mod}, in the sense that we only need local information above the conductor to understand the isomorphism class.
\begin{prop}\label{prop:coprime_cond}
  Let~$S$ be an order in~$K$ and let~$r$ be a non-zero divisor of~$K$.
  Assume that~$r S$ is coprime to the conductor~$\frf=(S:\cO_K)$.
  Then for every fractional~$S$-ideal~$I$ we have an~$S$-linear isomorphism
  \[ \frac{I}{rI} \cong \frac{S}{r S}.\]
\end{prop}
\begin{proof}
  Because~$rS$ is coprime to the conductor, 
  by Lemma \ref{lemma:rel_cond_same_order} we have~$S_\p=\cO_{K,\p}$ for any prime~$\p$ containing~$r$.
  This implies that~$I_\p\cong S_\p$ for every such prime~$\p$.
  We conclude by Lemma~\ref{lemma:loc_glob_fin_mod}.
\end{proof}

Given a lattice~$L$ in~$K$ we define its \emph{trace dual} as
\[ L^t=\set{ x\in K : \Tr_{K/Q}(x L) \subseteq Z }. \]
In the following lemma, we record some well known properties of trace duals.
\begin{lemma}\label{lemma:trace}
  Let~$L$,$L_1$ and~$L_2$ be lattices and let~$S$ be an order in~$K$.
  Then
  \begin{enumerate}[(a)]
    \item \label{lemma:trace:doubledual}~$(L^t)^t=L$.
    \item~$(L_1:L_2) = (L_1^tL_2)^t~$.
    \item~$(L_1:L_2) = (L_2^t : L_1^t)$.
    \item \label{lemma:trace:mult_ring}~$(L:L) = S$ if and only if~$LL^t = S^t$.
  \end{enumerate}
\end{lemma}
\begin{proof}
  See for example \cite[Sec.~15.6]{JV}.
\end{proof}

Let~$S$ be an order,~$\p$ be a prime of~$S$ and~$p$ its contraction in~$Z$.
Denote by~$\hat{Q}_p$ the fraction field of~$\hat{Z}_p$ and by
$\hat{K}_\p$ the total ring of quotient of~$\hat{S}_\p$. 
The trace~$\Tr_{K/Q}$ naturally induces a trace~$\Tr_{\hat{K}_\p/\hat{Q}_p}$.
It follows that taking trace duals commutes with completion.
So, given a fractional~$S$-ideal~$I$, the notation~$\hat{I}^t_\p$ is not ambiguous.

\begin{lemma}\label{lemma:loc_dual}
  Let~$S$ be an order and~$\p$ be a prime of~$S$. For fractional~$S$-ideals~$I$ and~$J$, we have that~$I_\p\cong J_\p$ as~$S_\p$-modules if and only if~$(I^t)_\p\cong (J^t)_\p$.
\end{lemma}
\begin{proof}
  By \cite[Ex.~7.5, p.~203]{eis95}, it is enough to prove the equivalent statement with completion instead of localization.
  Assume that~$\hat I_\p \cong \hat J_\p$, that is,~$\hat I_\p = \alpha \hat J_\p$ for some~$\alpha\in \hat K^\times_\p$.
  Then~$\hat I^t_\p = \alpha^{-1} \hat J_\p^t$.
  Hence~$\hat I^t_\p \cong \hat J_\p^t$, as required.
  We used here that completions and trace duals commute, as noted above.
  The converse direction follows from Lemma~\ref{lemma:trace}.\ref{lemma:trace:doubledual}.
\end{proof}

Recall that Lemma~\ref{lemma:loc_glob_fin_mod} provides an~$S$-linear isomorphism of quotients of fractional ideals, under certain local conditions.
However, when comparing a fractional ideal and its dual, we can obtain a~$Z$-linear isomorphism between quotients, regardless of the local behavior.
\begin{lemma}\label{lemma:matlisduality}
  Let~$S$ be an order with fractional ideals~$I$ and~$J$, and let~$r\in S$ be a non-zero divisor.
  Then we have a~$Z$-linear isomorphism 
  \[ \frac{I}{rI} \cong \frac{I^t}{rI^t}. \]
\end{lemma}
\begin{proof}
  This is a special case of \cite[Lem.~2.4.(iv)]{MarsegliaCMType_arxiv}, which is an application of Matlis duality; see for example \cite[Thm.~1.7]{Ooishi76}.
\end{proof}

We now recall some properties of orders that were studied in \cite[Sec.~3]{MarsegliaCMType_arxiv}.
The \emph{Cohen-Macaulay type} of an order~$S$ at a prime~$\p$, denoted by~$\type_\p(S)$, is the minimal number of generators of~$(S^t)_\p$ as an~$S_\p$-
This definition is equivalent to the usual one; see~\cite[Sec.~3]{MarsegliaCMType_arxiv}.
We say that an order~$S$ is \emph{Gorenstein} at a prime~$\p$ if its Cohen-Macaulay~$\type_\p(S)$ is equal to~$1$, that is, if~$(S^t)_\p$ is a principal~$S_\p$-module.
We say that~$S$ is Gorenstein if it is so at every prime.
This definition of Gorenstein is equivalent to the ones typically used in the literature. For example, a ring with finite (Krull) dimension is called Gorenstein if it has finite injective dimension. See \cite[Sec.~1]{basshy63} for other equivalent definition, and see \cite[Thm.~6.3]{basshy63} and \cite[Prop.~2.7]{buchlenstra} for the proof of the equivalence with the one used in this paper.
In fact, using the latter reference, one can deduce the following lemma.
We give a complete proof for convenience.
\begin{prop}\label{prop:Gor_at_p}
  Let~$S$ be an order and~$\p$ be prime of~$S$.
  Then~$S$ is Gorenstein at~$\p$ if and only if every fractional~$S$-ideal~$I$ with~$(I:I)_\p=S_\p$ is principal at~$\p$.
\end{prop}
\begin{proof}
  Observe that~$S$ is Gorenstein at~$\p$ if and only if~$\hat S^t_\p = \alpha \hat S_\p$ for some~$\alpha\in \hat{K}^\times_\p$.
  Assume now that~$S$ is Gorenstein at~$\p$.
  Pick a fractional~$S$-ideal~$I$ with~$(I:I)_\p=S_\p$.
  By taking the completion we get that~$\widehat{(I:I)}_\p=\hat S_\p$.
  By Lemma \ref{lemma:trace}.\ref{lemma:trace:mult_ring}, we obtain that~$\hat I_\p \hat I^t_\p=\hat S^t_\p = \alpha \hat S_\p$.
  Hence,~$\hat I_\p$ is invertible.
  Because invertible fractional~$\hat S_\p$-ideals are principal, we get that~$I_\p$ is a principal~$S_\p$-module, as required.
  For the converse, it is enough to observe that~$S^t$ has multiplicator ring~$S$.
\end{proof}

Similarly to the Gorenstein case, locally, we have a classification of fractional ideals with multiplicator ring of type~$2$.
\begin{prop}[{\cite[Thm.~6.2]{MarsegliaCMType_arxiv}}]\label{prop:cmtype2atP}
  Let~$S$ be an order and~$\p$ a prime of~$S$ such that
 ~$\type_\p(S)=2$.
  Then for every fractional~$S$-ideal~$I$ such that~$(I:I)_\p=S_\p$ then either~$I_\p\cong S_\p$ or~$I_\p\cong (S^t)_\p$.
\end{prop}
We exploit the classifications of Proposition~\ref{prop:Gor_at_p} and Proposition~\ref{prop:cmtype2atP} to understand quotient of fractional ideals, as we show in the next proposition.
We invite the reader to compare the statement with Proposition~\ref{prop:coprime_cond}, where the isomorphism is~$S$-linear, while here we only get a~$Z$-linear isomorphism.
\begin{prop}\label{prop:cmtype_at_most_2_at_p}
  Let~$r\in S$ be a non-zero divisor.
  If
 ~$\type_\p(S)\leq 2$ 
  for all primes~$\p$ of~$S$ containing~$r$ then for all fractional~$S$-ideal~$I$ with~$(I:I)_\p=S_\p$ we have a~$Z$-linear isomorphism
  \[ \frac{I}{rI} \cong \frac{S}{r S}.\]
\end{prop}
\begin{proof}
  Fix~$\p$ containing~$r$.
  If~$S$ is Gorenstein at~$\p$ then~$I_\p\cong S_\p$ by Proposition~\ref{prop:Gor_at_p}.
  If 
 ~$\type_\p(S)=2$ then
 ~$I_\p \cong S_\p$ or~$I_\p \cong (S^t)_\p$
  by Proposition \ref{prop:cmtype2atP}.
  Set~$M=I/rI$ and~$N=S/r S$.
  If~$I_\p \cong S_\p$ then we have an induced~$S_\p$-linear isomorphism
  \[ M_\p \cong N_\p. \]

  If~$I_\p \cong (S^t)_\p$ then first we observe that~$\hat{I}_\p\cong \hat{S}^t_\p$.
  So we have an induced~$\hat{S}_\p$-linear isomorphism
  \[ \hat{M}_\p \cong \frac{\hat S^t_\p}{r \hat S^t_\p},\] 
  which combined with Lemma~\ref{lemma:matlisduality} gives a~$\hat{Z}_p$-linear isomorphism
  \[ \hat{M}_\p \cong \hat{N}_\p. \]
  Since~$M_\p$ and~$N_\p$ are finitely generated~$Z_p$-modules,
  we obtain a~$Z_p$-linear isomorphism
  \[ M_\p\cong N_\p \]
  by \cite[Ex.~7.5, p.~203]{eis95}.

  By Lemma~\ref{lem:dir_sum_loc}, we have
  \[ M \cong \bigoplus_{\p} M_\p \text{ and } N \cong \bigoplus_{\p} N_\p, \]
  where the direct sums run over the primes~$\p$ containing~$r$, since~$M_\p=N_\p=0$ for all other primes.
  We conclude that~$M\cong N$ as~$Z$-modules.
\end{proof}

As previously anticipated, we will apply the results contained in this section to orders in commutative endomorphism algebras of abelian varieties over finite fields. 
Such algebras have an automorphism that corresponds to the Rosati involution, and acts as complex conjugation.
Putting together previously stated results with this extra structure, we obtain the following proposition, which will be used to prove that certain abelian varieties are not self-dual in Proposition~\ref{prop:not_self_dual}.

\begin{prop}\label{prop:cmtype2notselfdual}
  Assume that~$K$ has an involution~$x\mapsto\bar x$ which fixes~$Q$ point-wise.
  Let~$S$ be an order in~$K$ satisfying~$S=\bar{S}$.
  Let~$\p$ be a prime of~$S$ such that $\type_\p(S)=2$.
  Then~$\p=\bar\p$ if and only if all fractional~$S$-ideals~$I$ with~$(I:I)_\p=S_\p$ satisfy~$I_\p \not \cong (\bar{I}^t)_\p$.
\end{prop}
\begin{proof}
  Assume that~$\p=\bar\p$. 
  By Proposition \ref{prop:cmtype2atP} we have that~$I_\p\cong S_\p$ or~$I_\p\cong (S^t)_\p$.
  Assume the former. 
  By Lemma~\ref{lemma:loc_dual}, we obtain 
  \[ (\bar{I}^t)_\p = (\bar{I}^t)_{\bar\p}\cong (\bar S^t)_{\bar \p} = (S^t)_\p. \]
  Similarly, if~$I_\p\cong (S^t)_\p$ then
  \[ (\bar{I}^t)_\p = (\bar{I}^t)_{\bar\p} \cong \bar S_{\bar \p} = S_\p. \]
  In both cases, if~$I_\p\cong (\bar I^t)_\p$ then~$S^t$ is principal at~$\p$, that is,~$S$ is Gorenstein at~$\p$, which is a contradiction.
  
  Now assume that~$\p \neq \bar \p$.
  Let~$d\in K^\times$ such that~$dS^t\subseteq S$ and~$m>0$ such that $\bar \p^m S_{\bar \p} \subseteq (dS^t)_{\bar \p}$.
  Consider the fractional~$S$-ideal~$I$ defined as:
  \[ I = dS^t +\bar\p^m. \]
  For every~$\mathfrak{l}\neq \bar \p$ we have
  \[
      I_\mathfrak{l} = (dS^t)_\mathfrak{l} + S_\mathfrak{l} = S_\mathfrak{l},
  \]
  and
  \[
    I_{\bar\p} = (dS^t)_{\bar\p} +\bar\p^m S_{\bar\p} = (dS^t)_{\bar\p}.
  \]
  It follows by Lemma~\ref{lemma:loc_dual} that~$(\bar I^t)_\p \cong \bar S_\p = S_\p$, which gives us~$I_\p\cong (\bar I^t)_\p$.
  Moreover, we see that~$(I:I)=S$ by checking the equality locally at every prime.
\end{proof}
\begin{remark}
  Note that in the proof of Proposition~\ref{prop:cmtype2notselfdual} we showed that if~$\p\neq\bar\p$ then there exists a fractional ideal~$I$ such that~$I_\p \cong (\bar{I}^t)_\p$ with multiplicator ring~$S=(I:I)$ globally not only locally at~$\p$.
\end{remark}

%%%%%
%Groups of rational points
%%%%%
\section{Groups of rational points and Gorenstein orders}\
\label{sec:gor}

Our goal is to understand groups of rational points~$A(\F_q)$ for abelian varieties~$A$ defined over~$\F_q$. 
To accomplish this goal in practice, it is productive to view~$A(\F_q)$ as not merely a group, but as a module over its endomorphism ring~$\End_{\F_q}(A)$. Although requesting a description of the additional structure may appear to make the problem harder \emph{a priori}, the module structure can be exploited and cleanly described in many cases, which allows one to deduce the group structure immediately.

Given a separable endomorphism~$s\colon A\to A$ of an abelian variety~$A$ over~$\F_q$, we denote by~$A[s]$ the~$\bar\F_q$-points of the kernel of~$s$.
\begin{prop} 
  \label{prop:norm_deg}
  If~$A$ is a squarefree abelian variety over~$\F_q$ and~$s$ is a separable endomorphism of~$A$, then~$\#A[s]=\deg(s) = N_{K/\Q}(s)$ where~$K = \End(A)\otimes_\Z\Q$.
\end{prop}
\begin{proof} 
  Write~$A \sim B_1\times \ldots \times B_r$ as the product of simple pairwise non-isogenous varieties. 
  Then~$K = \End(A)\otimes_\Z\Q \cong \prod_{i = 1}^r K_i$ is the product of the number fields~$K_i = \End(B_i)\otimes_\Z\Q$.
  Without loss of generality, by \cite[Thm.~3.13]{Wat69}, we can choose each~$B_i$ so that~$\End(B_i)$ is the maximal order~$\cO_{K_i}$ in~$K_i$. 
  Since~$\End(A) \subseteq \prod_{i = 1}^r\mathcal{O}_{K_i}$, we can write~$s=(s_1,\ldots, s_r)$ for~$s_i\in\cO_{K_i}$.
  It is therefore enough to consider the case where~$A = B_1 \times \ldots \times B_r$.
  When~$A$ is simple, i.e., when~$r=1$, the statement follows from \cite[Prop~12.12]{Milne86}. Therefore, by definition, 
 ~$$
    N_{K/\Q}(s) 
  = \prod_{i = 1}^r N_{K_i/\Q}(s_i)
  =\prod_{i = 1}^r \deg(s_i)
  =\deg(s),
 ~$$ 
  where the final equality follows from the fact that $\deg(s) = \#A[s]$ and~$\deg(s_i) = \#B_i[s_i]$ by separability.
\end{proof}
Theorem \ref{thm:Gorenstein_general} and Corollary \ref{cor:Gorenstein} below were proven in the case where~$A$ is an elliptic curve by Lenstra \cite{Lenstra96}, and generalized to the case of simple abelian varieties in~\cite{Springer21}.
To obtain the same result for square-free varieties, we follow the proof method used in the simple case.

\begin{thm}\label{thm:Gorenstein_general}
Let~$A$ be a squarefree abelian variety over~$\F_q$ and let~$s$ be a separable endomorphism of~$A$.
If~$\End(A)$ is Gorenstein at the primes containing~$s$, then 
$$
	A[s] \cong \frac{\End(A)}{s\cdot\End(A)}
$$ 
is an isomorphism of~$\End(A)$-modules. 
\end{thm}

\begin{proof} 
  Write~$ S = \End(A)$, ~$S_0 = S/Ss$ and~$M = A[s]$. 
  By the universal property of quotients, if~$rA[s] = 0$, then~$r = ts$ for some~$t\in S$. 
  This implies that~$M$ is a faithful~$S_0$-module. 

  Moreover, because~$S$ is Gorenstein at every prime ideal containing~$s$ and~$s$ is a non-zero-divisor, we deduce that~$S_0$ is a finite Gorenstein ring by \cite[Ex.~18.1]{Matsumura86}.
  Therefore,~$M$ contains  a free~$S_0$-submodule~$N$ of rank 1 by \cite[Lem.~2.3]{Springer21}.
  By Proposition~\ref{prop:norm_deg}, and the fact that the modules are finite, we have
  \begin{align*}
    \#M = \deg s 
      = N_{K/\Q}(s) 
      = \#(S/Ss) 
      = \#S_0.
  \end{align*} 
  We deduce that~$M = N \cong S/Ss$, as desired.
\end{proof}

\begin{cor}\label{cor:Gorenstein}
  Let~$A$ be a squarefree abelian variety over~$\F_q$ and let~$\pi$ be the Frobenius endomorphism of~$A$. 
  If~$n\geq 1$ and~$\End(A)$ is Gorenstein at the prime ideals containing~$1-\pi^n$, then there is an isomorphism of~$\End(A)$-modules 
 ~$$
    A(\F_{q^n}) \cong 
    \frac{\End(A)}{(1-\pi^n)\End(A)}.
 ~$$ 
  In particular,~$B(\F_{q^n}) \cong A(\F_{q^n})$ are isomorphic groups for all abelian varieties~$B$ which are isogenous to~$A$ with~$\End(B) = \End(A)$.
\end{cor}
\begin{proof}
  We may apply Theorem \ref{thm:Gorenstein_general} because~$A(\F_{q^n}) = A[1-\pi^n]$ and~$1-\pi^n$ is a separable isogeny.
\end{proof}
However, note that these results are not true in general when we remove the assumption that~$\End(A)$ is Gorenstein, see Example~\ref{ex:diff_grps_same_end}.

%%%%%
%CYCLIC VS CONDUCTOR
%%%%%
\section{Cyclic isogeny classes}
\label{sec:coprime_conductor}
In this section we study isogeny classes containing only abelian varieties with cyclic groups of rational points.
\begin{df} 
  We say that an isogeny class is \emph{cyclic} if every variety~$A$ in the isogeny class has a cyclic group~$A(\F_q)$ of rational~$\F_q$-points.
\end{df}
 Any isogeny class of abelian varietes with a single point is trivially cyclic.
 In Section \ref{sec:red_sqfree}, 
 we show that, except for such trivial factors, every cyclic isogeny class is squarefree.
Then, in Section \ref{sec:sqfree_cyclic}, we provide a characterization of precisely which squarefree isogeny classes are cyclic in terms of conductor ideals.

\subsection{Reducing to the squarefree case}
\label{sec:red_sqfree}
In this subsection, we consider abelian varieties whose endomorphism algebras are non necessarily commutative.
In general,  if~$B$ is an abelian variety over~$\F_q$, then~$\End(B)$ is an order in the endomorphism algebra~$K = \End(B)\otimes_\Z\Q$, that is a subring which is finitely generated free module over~$\Z$, and whose~$\Q$-span is the whole algebra~$K$. 
Observe that this notion of order specializes to the one introduced in Section~\ref{sec:frac_ideals} in the commutative case.

\begin{prop} 
  \label{prop:simple_cyclic_sqfree}
  If~$B$ is a simple abelian variety over~$\F_q$ such that~$B(\F_q)$ is a nontrivial cyclic group and~$\End(B)$ is a maximal order in~$\End(B)\otimes_\Z \Q$, then~$\End(B)$ is commutative, that is,~$B$ is squarefree.
\end{prop}
\begin{proof} 
  Let~$L$ be the center of the endomorphism algebra~$\End(B)\otimes_\Z \Q$. 
  Because~$B$ is simple,~$L = \Q(\pi)$ where~$\pi$ is a root in~$L$ of the polynomial~$h_B(x)$; see \cite[Thm.~8]{MilWat71}.
  The center of~$\End(B)$ is~$\cO_L$ by the maximality of~$\End(B)$. 
  By \cite[Thm.~1.3.(b)]{Springer21}, there is an isomorphism of~$\cO_L$-modules
 ~$$
    B(\F_q) \cong \left(\frac{\cO_L}{(1-\pi)\cO_L}\right)^{d},
 ~$$
  where~$d = 2\dim(B)/[L : \Q]$. 
  In particular, because the group of points is nontrivial and cyclic by hypothesis, we must have~$d = 1$, which is equivalent to $\cO_L = \End(B)$ by \cite[Thm.~8]{MilWat71}.
\end{proof}

\begin{thm} 
  \label{thm:general_cyclic_sqfree}
  If~$A$ is an abelian variety over~$\F_q$ whose isogeny class is cyclic, then~$A \sim  A_1 \times A_{\text{sf}}$ for abelian varieties~$A_1$ and~$A_{\text{sf}}$ over~$\F_q$, possibly of dimension~$0$, such that~$A_{\text{sf}}$ is squarefree and~$\#A_1(\F_q) = 1$. If~$q \geq 5$, then~$A$ itself is squarefree.
\end{thm}
\begin{proof}
  Decompose the isogeny class ~$A\sim B_1^{e_1}\times \dots\times B_r^{e_r}$ into distinct simple factors.
  By \cite[Theorem 3.13]{Wat69}, we may assume that each~$\End(B_j)$ is a maximal order in its endomorphism algebra. 

  After possibly reordering, let~$r_1$ be such that~$\#B_j(\F_q) = 1$ for all~$1\leq j \leq r_1$ and~$\#B_j(\F_q) > 1$ for all~$r_1 < j \leq r$.
  Define 
 ~$$
    A_1 = \prod_{1 \leq j \leq r_1}B_j^{e_j} \text{ and } A_{\text{sf}} = \prod_{r_1 < j \leq r}B_j^{e_j}.
 ~$$
  It suffices to show~$A_{\text{sf}}$ is square-free. In this case,~$e_j  =1$ for all~$r_1< j \leq r$ because the isogeny class is cyclic, and Proposition \ref{prop:simple_cyclic_sqfree}
  shows that
 ~$$\End(A_{\text{sf}}) \cong  \prod_{r_1 < j \leq r }\End(B_j)$$
  is commutative, that is,~$A_{\text{sf}}$ is squarefree.
  The theorem then follows because the Weil bound states~$\#B(\F_q) \geq (\sqrt q - 1)^{2\dim(B)}$ for any abelian variety~$B$ over~$\F_q$; see \cite{Weil48}.
  Thus~$\#B(\F_q) = 1$  implies~$q \leq 4$.  
\end{proof}
  
\begin{remark} 
  \label{rem:one_point}
  There are infinitely many simple abelian varieties over~$\F_2$ with a single rational point by \cite{MP77}; see also \cite{Ked21}.
  However, there is only one simple isogeny class over each of~$\F_3$ and~$\F_4$ with a single rational point; see \cite[Theorem 3.2]{Kadets21} and LMFDB isogeny classes 
  \href{http://www.lmfdb.org/Variety/Abelian/Fq/1/3/ad}{1.3.ad}
  and
 \href{http://www.lmfdb.org/Variety/Abelian/Fq/1/4/ae}{1.4.ae}
  \cite{LMFDB}. 
\end{remark}

\subsection{The cyclicity of squarefree abelian varieties}
\label{sec:sqfree_cyclic}
We will use the same notation as in the previous sections.
Let~$\cI_h$ be a squarefree isogeny class of abelian vareities over~$\F_q$.
Put~$K=\Q[x]/(h)=\Q[\pi]$ and~$R=\Z[\pi,\bar{\pi}]$. 

\begin{thm}\label{thm:coprime_iff_cyclic}
  The isogeny class~$\cI_h$ is cyclic if and only if~$(1-\pi)R$ is coprime to the conductor~$\frf = (\cO_K : R)$, that is,~$(1-\pi)R + \frf = R$.
\end{thm} 
This theorem is a straightforward combination of Proposition~\ref{prop:coprime_cyclic} and Corollary~\ref{cor:cyclicity_converse}. 
We remark that the following is a simple application of Theorem \ref{thm:coprime_iff_cyclic}.
\begin{cor}
    If~$h(1)$ is a square-free integer, where~$h(t)$ is the characteristic polynomial of~$\pi$,
    then~$(1-\pi)R$ is coprime with~$\frf=(R:\cO_K)$.
\end{cor}
\begin{proof}
    Because~$h(1)$ is square-free, any finite abelian group of order~$h(1)$ is cyclic. Thus, the isogeny class is cyclic, and we conclude via Theorem \ref{thm:coprime_iff_cyclic}.
\end{proof}
%%%%%
%Case: Coprime
%%%%%

The next proposition proves one direction in Theorem~\ref{thm:coprime_iff_cyclic}.
\begin{prop}\label{prop:coprime_cyclic}
  If~$(1-\pi)R$ is coprime to the conductor~$\frf = (R:\cO_K)$, then, for every~$A$ in~$\cI_h$, we have an~$R$-linear isomorphism
  \[A(\F_q)\cong \frac{R}{(1-\pi)R}.\]
  In particular, the isogeny class~$\cI_h$ is cyclic.
\end{prop}
\begin{proof}
  Let~$A$ be an abelian variety in~$\cI_h$.
  Put~$S=\End(A)$, and note that~$R\subseteq S$.
  Hence~$(1-\pi)S$ is coprime to the conductor~$(S:\cO_K)$ of~$S$ in~$\cO_K$.
  For every prime~$\p$ of~$S$ containing~$(1-\pi)$ we have~$S_\p=\cO_{K,\p}$ by Lemma~\ref{lemma:rel_cond_same_order}.
  Therefore~$S$ is Gorenstein at~$\p$ for every~$\p$ containing~$(1-\pi)$.
  By Theorem~\ref{thm:Gorenstein_general} we have that
  \[ A(\F_q) \cong \frac{S}{(1-\pi)S}. \]
  Since~$S$ is a fractional~$R$-ideal, Proposition~\ref{prop:coprime_cond} gives us an~$R$-linear isomorphism $S/(1-\pi)S\cong R/(1-\pi)R$.
  We conclude by observing that 
  \[ \frac{R}{(1-\pi)R} \cong  \frac{\Z[x,y]}{(h(1),x-1,y-q)} \cong \frac{\Z}{(h(1))},\]
  which is immediate from the method of the proof of \cite[Prop.~2.7]{MarSpr_pts_PAMS}.
\end{proof}

%%%%%
%Case: Not coprime
%%%%%

We now prove a strong converse to Proposition~\ref{prop:coprime_cyclic}.

\begin{prop}\label{prop:cyclicity_converse_general}
  Let~$A$ be a square-free abelian variety over~$\F_q$ with Gorenstein endomorphism ring~$S = \End(A)$ and Frobenius endomorphism~$\pi$.
  If~$(1-\pi)R$ is not coprime to the conductor~$\frf = (R : S)$, then~$A(\F_{q^n})$ is a non-cyclic~$R$-module for all~$n\geq 1$.
  In particular, every abelian variety in the isogeny class with endomorphism ring~$S$ has a non-cyclic group of points.
\end{prop}
\begin{proof} 
  We have~$A(\F_{q^n})\cong S/(1 - \pi^n)S$ as~$S$-modules by Corollary~\ref{cor:Gorenstein}. 
  Observe that~$(1- \pi^n)R$ is not coprime to~$\frf$
  because~$(1-\pi^n)=(1-\pi)(1+\pi+\ldots+\pi^{n-1})$ implies~$R \supsetneq (1-\pi)R + \frf \supseteq (1- \pi^n)R +\frf$ for all~$n\geq 1$.
  Therefore,~$A(\F_{q^n})$ is a non-cyclic~$R$-module by Proposition~\ref{prop:converse_O_K}.
\end{proof}

\begin{cor}\label{cor:cyclicity_converse}
  If~$(1-\pi)R$ is not coprime to the conductor~$\frf = (R : \cO_K)$, then~$A(\F_{q^n})$ is a non-cyclic~$R$-module for all~$n\geq 1$  for every~$A$ in~$\cI_h$ with maximal endomorphism ring.
  In particular, the isogeny class~$\cI_h$ is non-cyclic.
\end{cor}
\begin{proof}
  Observe that~$\cO_K$ is the endomorphism ring of an abelian variety in~$\cI_h$ by \cite[Theorem 3.13]{Wat69}.
  Now, apply Proposition~\ref{prop:cyclicity_converse_general} with~$S=\cO_K$, which is Gorenstein.
\end{proof}

%%%%%%
% Categorical equivalences, group of points and duals
%%%%%%
\section{Non-cyclic groups of rational points}\
\label{sec:non-cyclic}
In Section~\ref{sec:coprime_conductor}, we gave a characterization of isogeny classes in which every abelian variety has cyclic group of points.
In this section, we go in the opposite direction. In Theorem~\ref{thm:rich_condition}, we characterize isogeny classes~$\cI_h$ in which every abelian group of order~$h(1)$ occurs as the group of rational points, and we call such isogeny classes \emph{rich}; see Definiton~\ref{df:rich}. 

There are two main tools which we require to study rich isogeny classes.
The first is Lemma~\ref{lem:GiangrecoMaidana}. It generalizes a result by Giangreco-Maidana \cite[Lem.~2.1]{Giangreco-Maidana19} which was originally used to study cyclic isogeny classes.
We also use a theorem of Rybakov, recalled below as Theorem~\ref{thm:rybakov}, which 
provides a criterion for the existence of abelian varieties in a given squarefree isogeny class with a prescribed group of rational points. 

To conclude the section, we use Rybakov's theorem again to prove in Proposition~\ref{prop:exceptional_Fq} the existence of ordinary abelian varieties over~$\F_q$ whose~$\ell$-primary part has~$2$ generators whenever~$q \equiv 1 \bmod \ell$. This allows us to deduce an improved version of \cite[Thm.~3.3]{MarSpr_pts_PAMS}, which we present below as Theorem~\ref{thm:MS_improved}.

\subsection{Lemmas about characteristic polynomials}\
We provide some basic lemmas concerning minimal polynomials which we will apply to the polynomial~$h$ for a squarefree isogeny class~$\cI_h$.
For an element~$\alpha$ in an \'etale algebra~$K$ over~$\Q$, we write, respectively,~$h_\alpha(x)$ and~$m_{\alpha}(x)$ for the characteristic and minimal polynomials of the~$\Q$-linear map on~$K$ defined by multiplication by~$\alpha$.
Note that if~$K = \prod_{j = 1}^t K_j$ for number fields~$K_1, \dots, K_t$, and~$\alpha = (\alpha_1, \dots,\alpha_t)$, then we have $h_\alpha(x) = \prod_{j = 1}^t h_{\alpha_i}(x)$.
Observe that~$\alpha$ is in the maximal order~$\cO_K$ of~$K$ if and only if~$m_\alpha(x)$ has integer coefficients.

For a squarefree isogeny class~$\cI_h$, the previous notation applied to the \'etale algebra~$K = \Q[x]/(h)$ leads to~$h(x) = h_{\pi}(x) = m_{\pi}(x)$ where~$\pi \in K$ corresponds to the Frobenius endomorphism on any abelian variety in~$\cI_h$.

\begin{lemma}\label{lem:min_poly_mod}
    Let~$K$ be an \'etale algebra, let~$\alpha \in K^\times$, and let~$b,c\in \Q$, with~$b\neq 0$.
    Define~$r = \deg h_\alpha$.
    Then
    \[ h_{b\alpha+c}(x) = b^r\cdot h_\alpha\left(\frac{x}{b} - \frac{c}{b}\right) \]
    and
    \[ h_{1/\alpha}(x) = \frac{x^r}{h_\alpha(0)} \cdot h_{\alpha}(1/x). \]
    In other words, we recognize~$h_{1/\alpha}(x)$ as the reverse of the polynomial~$h_\alpha(x)/h_\alpha(0)$.
\end{lemma}
\begin{proof}
    First, we prove the statements when~$K = \Q(\alpha)$ is a number field. Both~$b^r\cdot m_\alpha\left(\frac{x}{b} - \frac{c}{b}\right)~$ and~$\frac{x^r}{m_\alpha(0)} \cdot m_{\alpha}(1/x)$ are monic, with coefficients in~$\Q$, of degree~$r$, and are~$0$ when evaluated at~$b\alpha+c$ and~$1/\alpha$ respectively.
    They are irreducible since we have isomorphisms
   ~$\Q(\alpha)\cong \Q (b\alpha+c) \cong \Q(1/\alpha).$

    In general, write~$K = \prod_{j = 1}^t K_j$ for number fields~$K_1, \dots, K_t$ and~$\alpha = (\alpha_1, \dots,\alpha_t)$.
    Setting~$d_j = [K_j : \Q(\alpha_j)]$, we observe that~$h_{\alpha}(x) = \prod_{j =1}^t m_{\alpha_j}(x)^{d_j}$.
    The claims follow from the previous case:
   ~$$
    h_{b\alpha+c}(x) = \prod_{j =1}^t m_{b\alpha_j+c}(x)^{d_j}
        = \prod_{j =1}^t b^{d_j\deg m_{\alpha_j}}\cdot m_{\alpha_j}\left(\frac{x}{b} - \frac{c}{b}\right)^{d_j} 
        =b^r\cdot h_\alpha\left(\frac{x}{b} - \frac{c}{b}\right);
   ~$$
   ~$$
      h_{1/\alpha}(x) = \prod_{j =1}^t m_{1/\alpha_j}(x)^{d_j}
        = \prod_{j =1}^t \frac{x^{d_j\deg m_{\alpha_j}}}{m_{\alpha_j}(0)} \cdot m_{\alpha_j}(1/x)^{d_j}
        =\frac{x^r}{h_\alpha(0)} \cdot h_{\alpha}(1/x).
   ~$$
\end{proof}

We apply the previous lemma to obtain a formula for the coefficients of the minimal polynomials of two related algebraic numbers.

\begin{lemma}\label{lem:coeff}
  Let~$\alpha$ be an element of an \'etale algebra~$K$ over~$\Q$ satisfying~$1-\alpha \in K^\times$, and let~$d\in \Q$.
  The coefficients of~$h_{\frac{d}{(1-\alpha)}}(x) = \sum_{i = 0}^r a_ix^i$
  are given by the formula
  \[ a_i = \frac{(-1)^{r+i}d^{r-i}h^{(r-i)}_\alpha(1)}{(r-i)!h_\alpha(1)}, \]
  where~$h^{(i)}_{\alpha}(x)$ is the~$i$-th derivative of the polynomial~$h_\alpha(x)$.
\end{lemma}
\begin{proof}
  We will use Lemma \ref{lem:min_poly_mod} several times.
	Let~$\beta=\frac{1}{d}(1-\alpha)$. We have
  \begin{equation}\label{lem:coeff:eq:1}
      h_{\frac{d}{(1-\alpha)}}(x) = h_{\frac{1}{\beta}}(x)=\frac{x^r}{h_\beta(0)} h_\beta\left( \frac{1}{x} \right);
  \end{equation}
  \begin{equation}\label{lem:coeff:eq:2}
      h_\beta(x) = \left( \frac{1}{d} \right)^r h_{1-\alpha}\left( dx \right);
  \end{equation}
  \begin{equation}\label{lem:coeff:eq:3}
      h_{1-\alpha}(x) = (-1)^r h_{\alpha}(1-x).
  \end{equation}
  Using Equations \eqref{lem:coeff:eq:2} with \eqref{lem:coeff:eq:3} we get
  \begin{equation}\label{lem:coeff:eq:4}
      h_\beta(0) = \left( \frac{1}{d} \right)^r h_{1-\alpha}\left( 0 \right) = (-1)^r\left( \frac{1}{d} \right)^r h_\alpha(1).
  \end{equation}
  Combining Equations \eqref{lem:coeff:eq:4} and \eqref{lem:coeff:eq:2} with Equation \eqref{lem:coeff:eq:1}, we obtain
  \begin{equation}\label{lem:coeff:eq:5}
      h_{\frac{d}{1-\alpha}}(x)
      =\frac{(-x)^r}{h_\alpha(1)} h_{1-\alpha}\left( \frac{d}{x} \right).
  \end{equation}
  Define~$h_{1-\alpha}(x) = x^r + b_{r-1}x^{r-1} + \ldots + b_{1}x+b_0$,
  and set~$b_r=1$.
  By Equation \eqref{lem:coeff:eq:3}, for every~$i=1,\ldots,r$ we get
  \begin{equation}\label{lem:coeff:eq:7}
      b_i = \frac{1}{i!}h^{(i)}_{1-\alpha}(x)\bigg\rvert_{x=0} = \frac{1}{i!}(-1)^{r+i}h^{(i)}_\alpha(1-x)\bigg\rvert_{x=0} = \frac{1}{i!}(-1)^{r+i}h^{(i)}_\alpha(1).
  \end{equation}
  Combining Equations \eqref{lem:coeff:eq:5}
  and \eqref{lem:coeff:eq:7}, we obtain
  \begin{align*}\label{lem:coeff:eq:8}
      h_{\frac{d}{(1-\alpha)}}(x)
          & = \frac{(-x)^r}{h_\alpha(1)}\left( \left(\frac{d}{x}\right)^r + b_{r-1} \left(\frac{d}{x}\right)^{r-1} + \ldots + b_1\left(\frac{d}{x}\right) + b_0 \right) \\
          & = \sum_{i=0}^r\frac{(-x)^r}{h_\alpha(1)}\left(\frac{d}{x}\right)^{r-i}b_{r-i}\\
          & = \sum_{i=0}^r\frac{(-x)^r}{h_\alpha(1)}\left(\frac{d}{x}\right)^{r-i}\cdot\frac{1}{(r-i)!}(-1)^{i}h^{(r-i)}_\alpha(1)\\
          & = \sum_{i=0}^r\frac{(-1)^{r+i}d^{r-i}h^{(r-i)}_\alpha(1)}{(r-i)!h_\alpha(1)}x^i.
  \end{align*}
  This concludes the proof.
\end{proof}

\subsection{Rich isogeny classes}\
In Section~\ref{sec:coprime_conductor}, we studied cyclic isogeny classes. Now we study the opposite extreme.

\begin{df}\label{df:rich}
  We call an isogeny class~$\cI_h$ \emph{rich} if every abelian group of order~$h(1)$ occurs as the group of rational points for some abelian variety in~$\cI_h$.
\end{df}

The following lemma detects the annihilator of the group of rational points in terms of the existence of certain endomorphisms.
It generalizes \cite[Lem.~2.1]{Giangreco-Maidana19}, which was originally presented for the sake of studying cyclic isogeny classes. 
We use our generalization to study rich isogeny classes instead.
\begin{lemma}\label{lem:GiangrecoMaidana}
    Let~$A$ be an abelian variety over~$\F_q$ with~$N$ rational points.
    Denote by~$\pi$ the Frobenius of~$A$.
    Let~$d$ be a divisor of~$N$.
    Then the following are equivalent:
    \begin{enumerate}[(a)]
        \item~$d A(\F_q) = 0~$.
        \item~$A(\F_q) \subseteq \ker\left(\left[d\right]\right)(\bar\F_q)$.
        \item There exists~$\vphi\in \End_{\F_q}(A)$ such that~$\left[d\right] = \vphi\circ (1-\pi)$.
        Moreover such~$\vphi$ lives in the center of~$\End_{\F_q}(A)$.
    \end{enumerate}
\end{lemma}
\begin{proof}
    The first two statements are clearly equivalent.
    To show that the second and the third are equivalent, we first observe that~$A(\F_q)=\ker(1-\pi)(\bar\F_q)$.
    In particular, it is clear that the third statement implies the second.
    Now, assume that the second statement holds, that is,~$\ker(1-\pi)(\bar\F_q)\subseteq \ker([d])(\bar\F_q)$.
    Then by the separability of~$1-\pi$, we have also~$\ker(1-\pi)\subseteq \ker([d])$.
    Consider the following commutative diagram
    \[
    \xymatrix{
                                & & A\\
        A \ar@{->>}[rr] \ar[urr]^{[d]} \ar[drr]_{1-\pi} &  & \dfrac{A}{\ker(1-\pi)} \ar[u]_{\tilde\vphi} \ar[d]^{\iota}\\
                                & & A
    }
    \]
    where the middle arrow is the canonical projection,~$\iota$ is the isomorphism induced by~$1-\pi$ and~$\tilde\vphi$ is the (unique) map induced by the inclusion~$\ker(1-\pi)\subseteq \ker([d])$ via the universal property of the quotient.
    Put~$\vphi = \tilde\vphi \circ\iota^{-1}$.
    Therefore
    \[ \vphi \circ (1-\pi)  = \left[ d \right]. \]
    Since both~$1-\pi$ and~$[d]$ are in the center and defined over~$\F_q$, the same holds for~$\vphi$, as required.
\end{proof}

Given two isogenous abelian varieties~$A$ and~$B$ over~$\F_q$, we identify the endomorphism algebras~$\End(A) \otimes_\Z \Q \cong\End(B) \otimes_\Z \Q$ with an isomorphism which maps the Frobenius endomorphism of~$A$ to the Frobenius endomorphism of~$B$.
We denote this element by~$\pi$.

\begin{prop}
  \label{prop:max_end_exp}
  Let~$A$ be an abelian variety over~$\F_q$ whose group of rational points~$A(\F_q)$ is annihilated by an integer~$d$. For every maximal order~$\cO$ in the endomorphism algebra~$\End(A)\otimes_\Z \Q$, there is an abelian variety~$B$ which is isogenous to~$A$ such that~$\End(B) \cong \cO$. Furthermore,~$B(\F_q)$ is also annihilated by~$d$ for all~$B$ with endomorphism ring~$\End(B) \cong \cO$.
\end{prop}
\begin{proof}
  The existence statement is \cite[Thm~3.13]{Wat69}.
  Because~$d$ kills~$A(\F_q)$, we see that~$d/(1-\pi)$ is an endomorphism of~$A$ which lies in the center of its endomorphism algebra by Lemma~\ref{lem:GiangrecoMaidana}.
  Because~$B$ has maximal endomorphism ring,~$\End(B)$ contains the maximal order of the center of its endomorphism algebra.
  In particular,~$d/(1-\pi)$ is an element of~$\End(B)$, which implies that~$d$ kills~$B(\F_q)$ by another application of Lemma~\ref{lem:GiangrecoMaidana}.
\end{proof}

We emphasize the fact that Lemma~\ref{lem:GiangrecoMaidana} and Proposition~\ref{prop:max_end_exp} apply to any abelian variety defined over a finite field. In the rest of the section, we restrict our attention to the squarefree case.

For a positive integer~$N$, we write~$\rad N$ for its radical, that is, if~$N = \prod_{j = 1}^s \ell_j^{e_j}$, then we define~$\rad N = \prod_{j = 1}^s \ell_j$. 
We say that~$N$ is the \emph{exponent} of an abelian group~$G$ if~$N$ is the smallest positive integer which annihilates~$G$.
The minimal possible exponent for a finite abelian group of order~$N$ is~$\rad N$, which is achieved by the ``least cyclic" group of order~$N$.
This notion is made precise by the following definition. 

Let~$g \geq 1$ be an integer, and let~$\ell$ be prime. Consider a group~$H$ with~$\ell$-primary part~$H_\ell = \prod_{j = 1}^{2g} (\Z/\ell^{e_j}\Z)$ where~$0\leq e_1\leq\dots\leq e_{2g}$. The~$\ell$-\emph{Hodge polygon} of~$H$ is the polygon with vertices~$(i,\sum_{j = 1}^{2g-i} e_j)$; see \cite[Def.~1.1]{Rybakov10} for details and examples.
We will use this notion when~$H$ is the group of rational points of an abelian variety~$A$, in which case~$g = \dim(A)$. 
Note that the group~$(\Z/\ell\Z)^f$ has the highest~$\ell$-Hodge polygon among abelian groups of order~$\ell^{f}$, while~$(\Z/\ell^f\Z)$ has the lowest.

Given a polynomial~$h(x) = \sum_{i =1}^{2g} a_i x^i \in \Z[x]$ with~$a_0 \neq 0$, the~$\ell$-\emph{Newton polygon} of~$h$ is the boundary of the lower convex hull of the points~$(i, \ord_\ell(a_i))$ for~$0\leq i\leq 2g$. 
The following theorem of Rybakov describes the groups of rational points occuring in squarefree isogeny classes by comparing Hodge polygons and Newton polygons.

\begin{thm}[{\cite[Thm.~1.1]{Rybakov10}}] \label{thm:rybakov}
  Given a squarefree isogeny class~$\cI_h$ of abelian varieties over~$\F_q$, a finite abelian group~$G$ of order~$h(1)$ occurs as the group of rational points of some abelian variety~$A$ in~$\cI_h$ if and only if the~$\ell$-Hodge polygon of~$G$ lies on or below the~$\ell$-Newton polygon of~$h(1-t)$ for all primes~$\ell$.
\end{thm}

We exploit this theorem to characterize which squarefree isogeny classes are rich.

\begin{thm} 
\label{thm:rich_condition}
Consider a squarefree isogeny class~$\cI_h$ of abelian varieties over~$\F_q$ of dimension~$g$.
Let~$K = \Q[x]/(h)$ be the endomorphism algebra, and let~$\pi$ be the class of~$x$.
Write~$N = h(1) = \prod_{j = 1}^s \ell_{j}^{e_j}$ for the number of rational points on each abelian variety in~$\cI_h$.
The following are equivalent.
\begin{enumerate}[(a)]
	\item~$\cI_h$ is rich, that is, every abelian group of order~$N$ arises as~$A(\F_q)$ for some~$A\in \cI_h$.\label{part:all_groups}
	\item There is an abelian variety~$A\in\cI_h$ whose group of rational points has exponent~$\rad{N}$, that is, 
  \[A(\F_q)\cong \prod_{j = 1}^s \left(\frac{\Z}{\ell_j\Z}\right)^{e_j}.\]
  \label{part:rad_n_group}
  \item The coefficients of the characteristic polynomial~$h_{\frac{\rad N}{1-\pi}}(x)$ are integers. \label{part:int_coeffs}
	\item For all~$1\leq i\leq 2g$, we have
	$$
		\frac{h^{(i)}(1)}{i!}\cdot \ell_1^{i-e_1}\cdots \ell_s^{i-e_s} \in \Z.
	$$\label{part:integral_fraction}
\end{enumerate} 
If one of the equivalent conditions holds, then~$A(\F_q)$ has exponent~$\rad N$ for every~$A$ in~$\cI_h$ with maximal endomorphism ring. 
\end{thm}
  \begin{proof} 
    Trivially, \ref{part:all_groups} implies \ref{part:rad_n_group}. The reverse direction is a consequence of Theorem~\ref{thm:rybakov} because, among abelian groups of order~$N$, the group of exponent~$\rad N$ has the highest~$\ell$-Hodge polygon for every prime~$\ell$.
    If~$A(\F_q)$ has exponent~$\rad N$ for some~$A\in \cI_h$, then~$\rad N/(1-\pi)$ is an element of~$\End(A)$ by Lemma \ref{lem:GiangrecoMaidana}. 
    In particular, ~$\rad N/(1-\pi)$ is an integral element, so its minimal polynomial has integer coefficients. Since~$K = \Q[\pi]$, the characteristic and minimal polynomials~$h_{\frac{\rad N}{1-\pi}}(x) = m_{\frac{\rad N}{1-\pi}}(x)$ coincide. It follows that \ref{part:rad_n_group} implies \ref{part:int_coeffs}.

    Conversely, assume that the minimal polynomial of~$\rad N/(1-\pi)$ has integer coefficients. 
    Then~$\rad N/(1-\pi)$ is contained in the maximal order~$\cO_K$ of the endomorphism algebra~$K = \Q[x]/(h)$.
    By  \cite[Thm~3.13]{Wat69}, there is always at least one abelian variety~$A\in \cI_h$ whose endomorphism ring~$\End(A)$ is the maximal order~$\cO_K$ in~$K=\End(A)\otimes_\Z\Q=\Q(\pi)$.
    Therefore, by Lemma~\ref{lem:GiangrecoMaidana},~$A(\F_q)$ is killed by~$\rad N$, so \ref{part:int_coeffs} implies \ref{part:rad_n_group}.

    Recall that~$h(x) = h_{\pi}(x)$. Applying Lemma \ref{lem:coeff} with~$d=\rad{N}$, we see that the polynomial~$h_{\frac{\rad N}{1-\pi}}(x)$ is in~$\Z[x]$ if and only if
    \[\frac{h^{(i)}(1)}{i!}\frac{\rad{N}^{i}}{N} = \frac{h^{(i)}(1)}{i!}\cdot \ell_1^{i-e_1}\cdots \ell_s^{i-e_s} \in \Z \]
    for~$i=1,\ldots,2g$. Hence parts \ref{part:int_coeffs} and \ref{part:integral_fraction} are also equivalent.

    The final claim about varieties with maximal endomorphism ring follows from combining Proposition~\ref{prop:max_end_exp} and part \ref{part:rad_n_group}.
  \end{proof}

As an application of Theorem~\ref{thm:rich_condition}, we show the following.

\begin{cor}
  \label{cor:rich_sqfree_simple}
  A squarefree isogeny class is rich if and only if its simple factors are rich.
\end{cor}
\begin{proof}
  Consider a squarefree isogeny class~$\cI_h$ over~$\F_q$ whose abelian varieties have~$N$ rational points. 
  Let~$A$ be an abelian variety in~$\cI_h$ with maximal endomorphism ring.
  By Theorem~\ref{thm:rich_condition}, the isogeny class~$\cI_h$ is rich if and only if the exponent of~$A(\F_q)$ is a squarefree integer, namely~$\rad N$. 
  
  Because~$\End(A)$ is maximal, ~$\End(A)\cong \prod_{j = 1}^r \cO_{K_j}$ where~$K_j = \Q[x]/(h_j)$ according to the factorization~$h = h_1\dots h_r$ into irreducible polynomials. 
  Moreover,~$A \cong A_1 \times \dots\times A_r$ where each~$A_j$ is simple and has maximal endomorphism ring~$\End(A_j) \cong \cO_{K_j}$.
  Observe that the exponent of~$A(\F_q) \cong \prod_{j = 1}^r A_j(\F_q)$ is the least common multiple of the exponents of~$A_j(\F_q)$ for~$1\leq j\leq r$. 
  Therefore, the exponent of~$A(\F_q)$ is a squarefree integer if and only if the exponent of~$A_j(\F_q)$ is a squarefree integer for all~$1\leq j\leq r$.
\end{proof}

  It is now straightforward to determine when a squarefree isogeny class is rich
  because Theorem~\ref{thm:rich_condition}.\ref{part:integral_fraction} provides a criterion which only involves computing the derivatives of the characteristic polynomial.
  Although Theorem~\ref{thm:coprime_iff_cyclic} provides a criterion for cyclcity in terms of conductor ideals, it is faster to use \cite[Thm.~2.2]{Giangreco-Maidana19}, which only requires computing one derivative of the characteristic polynomial.
  We exhibit some statistics over small finite fields in the following example.
  
  \begin{example}
  \label{ex:rich_vs_cyclic}
  Let~$\cI^{\text{sf}}(g,q)$ be the set of squarefree isogeny classes of dimension~$g$ over~$\F_q$, and let~$\cR(g,q)$ and~$\cC(g,q)$ be the subsets containing the rich and cyclic isogeny classes, respectively. 
   In Table~\ref{tab:rich_cyclic}, we collect statistics concerning the cardinalities of these sets for small values of~$g$ and~$q$. 
  
  \begin{table}[h!] \centering
  \begin{tabular}{|c|c|cccc|}
  \hline
   \multirow{2}{*}{$\F_q$} & \multirow{2}{*}{$g$} &~$\cR \setminus \cC$ &~$\cC \setminus \cR$ &~$ \cR\cap \cC$&~$\cI^{\text{sf}}\setminus (\cR \cup \cC)~$\\
 &  & Only Rich & Only Cyclic & Both & Neither\\ \hline
  \multirow{5}{*}{$\F_2$}
  & 1 & 0\% & 20.0\% & 80.0\% & 0\% \\
  & 2 & 3.45\% & 17.2\% & 75.9\% & 3.45\% \\
  & 3 & 8.66\% & 18.4\% & 66.0\% & 7.02\% \\
  & 4 & 10.5\% & 19.8\% & 61.0\% & 8.67\% \\
  & 5 & 10.7\% & 20.5\% & 58.6\% & 10.2\% \\
  & 6 & 10.0\% & 21.3\% & 58.5\% & 10.2\% \\
  \hline
  \multirow{4}{*}{$\F_3$}
  & 1 & 14.3\% & 0\% & 85.8\% & 0\% \\
  & 2 & 23.6\% & 5.45\% & 67.2\% & 3.64\% \\
  & 3 & 21.6\% & 8.38\% & 60.9\% & 9.17\% \\
  & 4 & 19.4\% & 9.31\% & 59.8\% & 11.5\% \\
  & 5 & 18.1\% & 9.83\% & 60.0\% & 12.1\% \\
  \hline
  \multirow{3}{*}{$\F_4$}
  & 1 & 0\% & 28.6\% & 71.4\% & 0\% \\
  & 2 & 11.9\% & 16.4\% & 61.2\% & 10.5\% \\
  & 3 & 13.1\% & 14.8\% & 60.1\% & 12.0\% \\
  & 4 & 12.9\% & 15.8\% & 60.3\% & 11.0\% \\
  \hline
  \multirow{3}{*}{$\F_5$}
  & 1 & 11.1\% & 11.1\% & 66.6\% & 11.1\% \\
  & 2 & 16.8\% & 9.25\% & 61.4\% & 12.6\% \\
  & 3 & 17.0\% & 10.4\% & 59.5\% & 13.1\% \\
  & 4 & 16.7\% & 10.4\% & 60.0\% & 12.9\% \\
  \hline
  \end{tabular} 
\caption{
  For~$2\leq q\leq 5$, we count the number of cyclic and rich squarefree isogeny classes of small dimension~$g$ over~$\F_q$ by applying \cite[Thm.~2.2]{Giangreco-Maidana19} and Theorem~\ref{thm:rich_condition} to data in the LMFDB \cite{LMFDB}. See also~\cite{LMFDB_paper}.
  \label{tab:rich_cyclic}
}
\end{table}
  An isogeny class is simultaneously rich and cyclic precisely when the properties are trivially satisfied, namely when the number of rational points~$N = \rad N$ is squarefree.  
  As a result, the intersection~$\cR(g,q) \cap \cC(g,q)$ can be considered the set of trivial examples.
  Asymptotically, the proportion of integers which are squarefree is~$6/\pi^2\approx 60\%$. 
\end{example}

\subsection{Groups with two generators}
To conclude this section, we consider the existence of abelian varieties whose groups of rational points are not cyclic, but are, locally, the product of only two cyclic factors.
As an application, we improve \cite[Thm.~3.3]{MarSpr_pts_PAMS} in the case of ordinary abelian varieties over~$\F_4$.

\begin{prop} 
\label{prop:exceptional_Fq}
  Let~$\ell$ be a prime and~$1\leq s_1\leq s_2$. If~$q\equiv 1 \bmod \ell^{s_1}$ is a prime power, then every squarefree isogeny class~$\cI_h$ over~$\F_q$ with~$\ord_{\ell}(h(1)) = s_1 + s_2$ contains an abelian variety~$A$ such that the~$\ell$-primary part of the group of rational points is
 ~$$
    A(\F_q)_\ell \cong (\Z/\ell^{s_1} \Z)\times (\Z/\ell^{s_2}\Z).
 ~$$
\end{prop}
\begin{proof} 
  The group ~$(\Z/\ell^{s_1} \Z)\times (\Z/\ell^{s_2}\Z)$ has~$\ell$-Hodge polygon defined by the points $(0, s_1+s_2), (1, s_1), (2,0)$.
  Therefore, using the notation~$h(1-x) = \sum_{j = 0}^{2g} b_j x^j$, it is enough show that~$\ord_\ell(b_0) \geq s_1+s_2$ and~$\ord_{\ell}(b_1)\geq s_1$ by Theorem~\ref{thm:rybakov}. The first holds by hypothesis because~$b_0 = h(1)$.

  By the~$q$-symmetry of~$h(x)$, we write
  \begin{align*}
  h(x) &= \left(\sum_{j = 1}^{g-1} a_{2g-j}(x^{2g - j} + q^{g-j} x^{j})\right) +  a_gx^g.
  \end{align*} 
  Because~$q\equiv 1\bmod \ell^{s_1}$, we deduce that~$2g - j + jq^{g-j} \equiv g(1 + q^{g-j})\bmod \ell^{s_1}$ for all~$0\leq j\leq g-1$. Combined with the fact that the polynomial~$(1-x)^j$ has linear coefficient~$-j$ for any~$j \geq 0$, we observe
  \begin{align*}
  -b_1 &= \left(\sum_{j = 1}^{g-1} a_{2g-j}(2g - j + jq^{g-j})\right) +  ga_g\\
  &\equiv \left(\sum_{j = 1}^{g-1} ga_{2g-j}(1 + q^{g-j})\right) +  ga_g\\
  &\equiv g h(1) \bmod \ell^{s_1}.
  \end{align*}
  Thus,~$\ord_\ell(b_1)\geq s_1$ because~$\ord_\ell(h(1))=s_1+s_2$, and we are done.
\end{proof}

\begin{cor} 
  \label{cor:divisible_by_4}
  Let~$q$ be an odd prime power. If~$\cI$ is a squarefree cyclic isogeny class over~$\F_q$, then its point count is not divisible by~$4$.
\end{cor}
\begin{proof} 
  If the point count is divisible by~$4$, then we apply Proposition \ref{prop:exceptional_Fq} with~$\ell = 2$ and~$s_1 = 1$ to find a non-cyclic variety.
\end{proof} 

\begin{cor}
\label{cor:exceptional_F4}
    For every~$s\geq 1$, the group~$\Z/3\Z\times \Z/3^s\Z$ arises as the group of rational points of a squarefree ordinary abelian variety over~$\F_4$.
\end{cor}
\begin{proof}
    Apply Proposition \ref{prop:exceptional_Fq} with~$\ell = 3$,~$q = 4$,~$s_1=1$ and~$s_2=s$.  The existence of ordinary isogeny classes with the desired number of points follows from Theorem 1.13 and Remark 1.16  in \cite{vBCLPS21}.
\end{proof}

For every~$N \geq 1$, there is an abelian variety~$A$ over~$\F_4$ with~$A(\F_4)\cong (\Z/N\Z)$, and~$A$ can be taken to be ordinary if and only if~$N \neq 3$. 
In particular, by taking products, this shows that every finite abelian group arises as the group of rational points of an abelian variety over~$\F_4$ which is not necessarily ordinary; see \cite[Thm.~3.3]{MarSpr_pts_PAMS}.
Corollary~\ref{cor:exceptional_F4} extends \cite[Thm.~3.3]{MarSpr_pts_PAMS} by proving that the abelian variety can be taken to be ordinary in additional non-cyclic cases.
We record the improved theorem here.

\begin{thm}
  \label{thm:MS_improved}
  Every finite abelian group~$G$ arises as the group of rational points $G \cong A(\F_q)$ for an abelian variety~$A$ over~$\F_4$. Moreover,~$A$ can be taken to be ordinary, except possibly if~$G = (\Z/3\Z)^n$ for an odd integer~$n$.
\end{thm}
\begin{proof}
  It is already established by \cite[Thm.~3.3]{MarSpr_pts_PAMS} that every finite abelian group occurs as the group of rational points of an abelian variety over~$\F_4$, and that the abelian variety can be taken to be ordinary except possibly when~$G$ takes the form~$(\Z/3\Z)^{n_1} \times \prod_{j >1} (\Z/3^j\Z)^{n_j}$, where~$n_1$ is odd. 
  Now consider one of the exceptional groups~$G$ where additionally~$n_s \geq 1$ for some~$s > 1$.
  There are ordinary abelian varieties~$A_1$ and~$A_2$ over~$\F_4$ whose groups of rational points are ~$$A_1(\F_4)\cong (\Z/3\Z)\times (\Z/3^{s}),$$~$$A_2(\F_4) \cong (\Z/3\Z)^{n_1 - 1}\times (\Z/3^s/\Z)^{n_s-1}\times \prod_{j\not\in \{1,s\}} (\Z/3^j\Z)^{n_j}$$
  by Corollary~\ref{cor:exceptional_F4} and \cite[Thm.~3.3]{MarSpr_pts_PAMS}, respectively.
  Therefore,~$G\cong (A_1\times A_2)(\F_4)$, as desired.
\end{proof}

%%%%%%
% Categorical equivalences, group of points and duals
%%%%%%
\section{Groups of rational points and categorical equivalences}\
\label{sec:ord-cs}
In this section, we first present Theorem~\ref{thm:cat_eq}, which describes groups of rational points by deploying a categorical equivalence between abelian varieties in certain squarefree isogeny classes and fractional ideals in \'etale algebras.
In such an isogeny class, we deduce in Proposition~\ref{prop:A_cmtype2_gp} that every abelian variety~$A$ with endomorphism ring~$\End(A) = S$ has the same group of rational points if~$S$ satisfies a certain local condition.
In Remark~\ref{rem:gor_vs_cmtype2}, we compare this result to Corollary~\ref{cor:Gorenstein}.

Throughout this section we use the usual notation.
We denote by~$\cI_h$ a squarefree isogeny class over~$\F_q$. 
We set~$K=\Q[x]/(h)$ and~$R=\Z[\pi,\bar \pi]$ where~$\pi$ is the class of~$x$ in~$K$.

\begin{df}
  We say that~$\cI_h$ satisfies
  \begin{itemize}
    \item {\bf Ord} if~$\cI_h$ is ordinary.
    \item {\bf CS} if~$q$ is prime.
  \end{itemize}
\end{df}
\begin{thm} 
  \label{thm:cat_eq}
  If~$\cI_h$ satisfies {\bf Ord} (resp.~{\bf CS}) then there exists a covariant (resp.~contravariant) equivalence between~$\cI_h$ and the category of fractional~$R$-ideals (with $R$-linear morphisms).
  Denote by~$\cF$ be the functor inducing the equivalence.
  Let~$A$ be an abelian variety in~$\cI_h$, with dual variety~$A^\vee$.
  Put~$\cF(A)=I$, where~$I$ is a fractional~$R$-ideal.
  \begin{enumerate}[(a)]
    \item \label{thm:cat_eq:Avee} We have~$\cF(A^\vee) = \bar{I}^t$ and~$\End(A^\vee)=\bar{\End(A)}$.
    \item \label{thm:cat_eq:grp_pts_A} There is a~$\Z$-linear isomorphism 
    \[ A(\F_{q^n}) \cong \frac{I}{(1-\pi^n)I},\]
    for all~$n \geq1$.
    \item \label{thm:cat_eq:grp_pts_Avee} There are~$\Z$-linear isomorphisms 
    \[ A^\vee(\F_{q^n}) 
      \cong \frac{\bar{I}^t}{(1-\pi^n)\bar{I}^t}
      \cong \frac{I^t}{(1-\bar \pi^n)I^t} 
      \cong \frac{I}{(1-\bar{\pi}^n)I} 
      \cong \frac{\bar I}{(1-\pi^n)\bar I} ,\]
    for all~$n \geq1$.
  \end{enumerate}
\end{thm}
\begin{proof}
  The existence of the equivalence is given by \cite{Del69} in the {\bf Ord} case, and by \cite{CentelegheStix15} in the {\bf CS} case.
  Part \ref{thm:cat_eq:Avee} is \cite[Thm.~5.2]{MarAbVar18} in the {\bf Ord} case, and \cite[Cor~3.26]{BergKarMar21} in the {\bf CS} case.
  Part \ref{thm:cat_eq:grp_pts_A}, for~$n=1$, is \cite[Cor.~4.7]{MarAbVar18} for 
  both cases, but the proof is identical for~$n>1$.
  In the {\bf Ord} case, the key ingredients are \cite[Lem.~4.13 and Prop.~4.14]{Howe95}, while the proof in the {\bf CS} case uses a local argument.

  For Part \ref{thm:cat_eq:grp_pts_Avee}, observe that
  the first~$\Z$-linear isomorphism is the combination of Parts \ref{thm:cat_eq:Avee} and \ref{thm:cat_eq:grp_pts_A}, while the second and fourth are the application of complex conjugation. 
  For the third isomorphism, we use Lemma~\ref{lemma:matlisduality} to deduce
  \[\frac{I^t}{(1-\bar{\pi}^n)I^t}
                      \cong \frac{(1-\bar{\pi}^n)^{-1}I}{I}
                      \cong \frac{I}{(1-\bar{\pi}^n)I}. \]
\end{proof}

\begin{remark}  
  We emphasize that the hypotheses for Theorem~\ref{thm:cat_eq} only impose conditions on the isogeny class over the base field~$\F_q$.
  Observe that property of being squarefree is not stable under base extension, and the functor we invoke in the {\bf CS} case requires the base field to be prime. 
  Nevertheless, we describe~$A(\F_{q^n})$ for all~$n\geq 1$, 
  because~$A(\F_{q^n}) = \ker(1-\pi^n)$ is the kernel of an isogeny defined over the base field~$\F_q$.
\end{remark}

\begin{remark}
  For~$n=1$, the~$\Z$-linear isomorphisms in Theorem~\ref{thm:cat_eq} are trivially~$R$-linear.
  Indeed, for part~\ref{thm:cat_eq:grp_pts_A} since~$R$ is generated over~$\Z$ by~$\pi$ and~$\bar \pi=q/\pi$, and~$I/(1-\pi)I$ is annihilated by~$(1-\pi)$,~$R$-linearity trivially follows from~$\Z$-linearity. 
  The same applies for part~\ref{thm:cat_eq:grp_pts_Avee}.
\end{remark}

Now we show that, in the {\bf Ord} and {\bf CS} cases, the group of rational points is uniquely determined by the endomorphism ring under certain conditions.

\begin{prop}\label{prop:A_cmtype2_gp}
  Let~$A$ be an abelian variety in a squarefree isogeny class~$\cI_h$ over~$\F_q$ satisfying {\bf Ord} or {\bf CS}.
  Write~$S = \End(A)$.
  For each~$n\geq1$, if~$\type_\p(S)\leq 2$
  for every prime~$\p$ of~$S$ above~$(1-\pi^n)$,
  then the group of~$\F_{q^n}$-rational points of~$A$ is uniquely determined by~$S$. 
  Specifically,
  \[ 
    A(\F_{q^n})\cong \frac{S}{(1-\pi^n) S}
  \]
 are isomorphic as~$\Z$-modules.
\end{prop}
\begin{proof}
  The statement follows from Proposition~\ref{prop:cmtype_at_most_2_at_p} with Theorem~\ref{thm:cat_eq}.
\end{proof}

\begin{remark}
  \label{rem:gor_vs_cmtype2}
  As noted in Section~\ref{sec:frac_ideals}, an order~$S$ is Gorenstein at a prime~$\p$ if and only if~$\type_\p(S)=1$.
  Hence Proposition~\ref{prop:A_cmtype2_gp} is a generalization of Corollary~\ref{cor:Gorenstein}.
  As a trade-off, the isomorphism is~$\Z$-linear rather than~$S$-linear, and we need further hypotheses on the isogeny class.
\end{remark}

In Proposition~\ref{prop:A_cmtype2_gp},
it is important that the local type of the endomorphism ring is at most 2 at primes above~$1-\pi^n$.
Indeed, in Example \ref{ex:diff_grps_same_end}, we produce abelian varieties with the same endomorphism ring of (local) type~$3$ but non-isomorphic groups of points. 

\begin{example}\label{ex:diff_grps_same_end}  
  The polynomial 
  \[ h=x^4 + 6x^2 + 25=(x^2 - 2x + 5)(x^2 + 2x + 5) \]
  defines an isogeny class~$\cI_h$ of ordinary abelian surfaces over~$\F_5$, which has label
  \href{http://www.lmfdb.org/Variety/Abelian/Fq/2/5/a_g}{2.5.a\_g} on the LMFDB \cite{LMFDB}.
  Consider the order~$S=\Z+2\cO_K$.
  This is the unique overorder of~$R$ with~$[\cO_K:S]=[S:R]=8$.
  Moreover,~$S$ is the unique overorder of~$R$ with a prime~$\p$ with~$\type_\p(S)=3$. This prime is~$\p=2\cO_K$, which is also the conductor of~$S$ in~$\cO_K$.
  
  Using Theorem \ref{thm:cat_eq} and algorithms from \cite{MarICM18}, we compute that there are~$5$ isomorphism classes of abelian varieties with endomorphism ring~$S$, represented by fractional ideals~$S$,$I$,$I^t$,$J$ and~$S^t$. One has~$J\cong J^t$.
  Let~$A_S,A_I,A_{I^t},A_J$ and~$A_{S^t}$ be the corresponding abelian varieties, respectively; see Theorem~\ref{thm:cat_eq}.
  We have
  \[ A_S(\F_5)\cong A_{S^t}(\F_5) \cong A_{J}(\F_5) \cong \frac{\Z}{2\Z}\times \frac{\Z}{2\Z} \times \frac{\Z}{8\Z},\] 
  and
  \[ A_I(\F_5)\cong A_{I^t}(\F_5) \cong \frac{\Z}{4\Z} \times \frac{\Z}{8\Z}. \]
\end{example}

\section{The dual abelian variety}\
\label{sec:dual}
In this section, we study the relationship between an abelian variety and its dual.
We use the categorical equivalences presented in Theorem~\ref{thm:cat_eq} which build a bridge between abelian varieties and fractional ideals. 
In the first part, we will prove that we have an isomorphism~$A(\F_q)\cong A^\vee(\F_q)$ under certain conditions on the endomorphism ring~$\End(A)$ and the fractional ideal associated to~$A$.
Clearly,~$A(\F_q)\cong A^\vee(\F_q)$ whenever~$A$ is the Jacobian of a curve or, more generally, a principally polarizable abelian variety. In fact, the implication only uses that~$A$ is a self-dual abelian variety, that is,~$A\cong A^\vee$.

In the second part of the section, we investigate when an abelian variety fails to be self-dual.
In particular, we prove that~$A$ is not self-dual if its endomorphism ring~$\End(A)$ satisfies a certain local condition.
To conclude, we provide a sequence of examples comparing various properties implying and implied by self-duality.

In this section, we use the same notation as Section~\ref{sec:ord-cs}.
Specifically, we denote by~$\cI_h$ a squarefree isogeny class over~$\F_q$. 
We set~$K=\Q[x]/(h)$ and~$R=\Z[\pi,\bar \pi]$ where~$\pi$ is the class of~$x$ in~$K$.

\subsection{The group of points of the dual abelian variety}\
In Proposition~\ref{prop:A_same_grp_Avee}, building on results proven previously, we give a list of conditions that guarantee the existence of an isomorphism~$A(\F_{q^n})\cong A^\vee(\F_{q^n})$.
In Example~\ref{ex:dual_non_isom_gps}, we exhibit an example of a geometrically simple ordinary abelian variety~$A$ with~$A(\F_{q})\not \cong A^\vee(\F_{q})$.
In Example~\ref{ex:stats_dual_non_isom_gps}, we show that squarefree ordinary examples over~$\F_q$ always exist in small dimensions for small finite fields~$\F_q$.

\begin{prop}\label{prop:A_same_grp_Avee}
  Let~$A$ be in~$\cI_h$ and put~$S=\End(A)$. Fix~$n\geq 1$.
  If one of the following assumptions holds,
  then~$A(\F_{q^n}) \cong A^\vee(\F_{q^n})$.
  \begin{enumerate}[(a)]
    \item \label{prop:A_same_grp_Avee:S_conj_stable_Gor}
   ~$S=\bar{S}$ and~$S$ is Gorenstein at~$\p$ for every prime~$\p$ of~$S$ above~$(1-\pi^n)$.
    \item \label{prop:A_same_grp_Avee:S_conj_stable_Gor_type_2}
   ~$\cI_h$ satisfies {\bf Ord} or {\bf CS},~$S=\bar{S}$, and~$\type_\p(S)\leq 2$ for every prime~$\p$ of~$S$ above~$(1-\pi^n)$.
    \item \label{prop:A_same_grp_Avee:local_isos}
         ~$\cI_h$ satisfies {\bf Ord} or {\bf CS}, and one of the following holds, where~$\cF(A)=I$.
          \begin{itemize}
            \item For every prime~$\p$ of~$R$ above~$(1-\pi^n)$ we have an~$R$-linear~$I_\p \cong (\bar{I})_\p$.
            \item For every prime~$\p$ of~$R$ above~$(1-\pi^n)$ we have an~$R$-linear~$I_\p \cong (\bar{I}^t)_\p$.
          \end{itemize}
  \end{enumerate}
\end{prop}
\begin{proof}
  Part \ref{prop:A_same_grp_Avee:S_conj_stable_Gor} 
  follows from Corollary~\ref{cor:Gorenstein}, because~$\End(A^\vee)=\bar S=S$.
  Part \ref{prop:A_same_grp_Avee:S_conj_stable_Gor_type_2} follows similarly from Proposition~\ref{prop:A_cmtype2_gp}.
  By Theorem~\ref{thm:cat_eq}, we have~$\Z$-linear isomorphisms
  \[ 
    A(\F_{q^n})\cong I\otimes_R\frac{R}{(1-\pi^n)R}
    \quad\text{and}\quad 
    A^\vee(\F_{q^n})\cong \bar I\otimes_R\frac{R}{(1-\pi^n)R}
                    \cong \bar I^t\otimes_R\frac{R}{(1-\pi^n)R}
  . \] 
  Combined with Lemma \ref{lemma:loc_glob_fin_mod}, this proves Part \ref{prop:A_same_grp_Avee:local_isos}.
\end{proof}
\begin{example}\label{ex:dual_non_isom_gps}
  Consider the isogeny class~$\cI_h$ of ordinary abelian surfaces over~$\F_4$ determined by the polynomial~$h=x^4 + 2x^3 + x^2 + 8x + 16$. 
  According to the LMFDB \cite{LMFDB}, this isogeny class, which has label \href{http://www.lmfdb.org/Variety/Abelian/Fq/2/4/c_b}{2.4.c\_b}, is geometrically simple and contains a Jacobian.
  Let~$\cO_K$ be the maximal order of~$K$.
  We have that~$2\cO_K = \p^2\bar\p^2$ where~$\p$ is a prime of~$\cO_K$.
  Consider the order~$S=R+\p^2$.
  It turns out that~$R$ has~$3$ overorders, namely,~$S$,$\bar S$ and~$\cO_K$, and all of these orders are Gorenstein.
  Using Theorem~\ref{thm:cat_eq}, there is an abelian variety~$A$ with~$\End(A) = S$, hence~$\End(A^\vee) = \bar S$.
  Using Corollary~\ref{cor:Gorenstein}, one computes that~$A(\F_4)$ and~$A^\vee(\F_4)$ are not isomorphic. Indeed, they are
  \[ \frac{\Z}{28\Z}\quad\text{and}\quad \frac{\Z}{2\Z}\times\frac{\Z}{14\Z}. \]
\end{example}
\begin{example}\label{ex:stats_dual_non_isom_gps}
  Consider the following set of pairs of positive intgers~$(g,q)$:
  \begin{align*}
    & \{ (2,q) : 2\leq q \leq 128 \text{ is a prime power } \}  \\
    &\cup \{ (3,q) : 2\leq q \leq 9 \text{ is a prime power }\} \cup \{ (3,16), (3,25) \} \\
    &\cup \{ (4,q) : q\in \{2,3,4\} \}  \\ 
    &\cup \{ (5,2) \}. 
  \end{align*}
  For each pair~$(g,q)$ in the set above, there is a squarefree ordinary abelian variety~$A$ of dimension~$g$ over~$\F_q$ satisfying
  \[ A(\F_q) \not \cong A^\vee(\F_q). \]
\end{example}

\subsection{Self-duality}
Recall the following well-known theorem.
\begin{thm} \label{thm:jac_ppav_implications}
  If~$A$ is an abelian variety over~$\F_q$, then each statement below implies the next:
  \begin{enumerate}[(a)]
    \item~$A$ is a Jacobian variety;\label{it:jac}
    \item~$A$ is a principally polarizable abelian variety;\label{it:ppav}
    \item~$A \cong A^\vee$ is self-dual;\label{it:sd}
    \item~$A(\F_q) \cong A^\vee(\F_q)$ are isomorphic groups. \label{it:gp_dual_same}
    \setcounter{enum_counter}{\value{enumi}}
    \addtocounter{enum_counter}{-1}
  \end{enumerate}
When~$A$ is squarefree, the following item~\ref{it:end_stable} is also implied by self-duality~\ref{it:sd}:
  \begin{enumerate}[(a')]
    \setcounter{enumi}{\value{enum_counter}}
    \item ~$\End(A) = \overline{\End(A)}$ is stable under complex conjugation.\label{it:end_stable}
  \end{enumerate}
\end{thm}

We already observed in Example~\ref{ex:dual_non_isom_gps} that properties~\ref{it:gp_dual_same} and \ref{it:end_stable} do not always hold.
Now we show that there are counterexamples to all of the reverse implications in Theorem~\ref{thm:jac_ppav_implications}.
For \ref{it:ppav}$\nRightarrow$\ref{it:jac}, see Example~\ref{ex:pp_not_jac}, and
for \ref{it:sd}$\nRightarrow$\ref{it:ppav}, see Example~\ref{ex:sd_not_pp}.
In Example~\ref{ex:same_end_not_sd}, we show that \ref{it:gp_dual_same} and \ref{it:end_stable} combined do not imply~\ref{it:sd}.
Moreover, in Example~\ref{ex:same_gp_not_same_end}, we show that \ref{it:gp_dual_same}$\nRightarrow$\ref{it:end_stable}, which is another exhibition that~\ref{it:end_stable}$\nRightarrow$\ref{it:sd}.
Note that we have the implication \ref{it:end_stable}$\Rightarrow$\ref{it:gp_dual_same} under certain hypotheses; see Proposition~\ref{prop:A_same_grp_Avee}.

The following example is well-known, but we record it for completeness.
\begin{example}[Principally polarizable but not Jacobian]
  \label{ex:pp_not_jac}
  It is easy to find principally polarizable varieties which are not Jacobians. For example, there are currently 30{\small,}079 geometrically simple ordinary isogeny classes on LMFDB which contain a principally polarizable abelian variety, but no Jacobian varieties; see \cite{LMFDB}.
\end{example}

\begin{example}[Self-dual but not principally polarizable]
  \label{ex:sd_not_pp}
  If~$\cI_h$ is a simple ordinary isogeny class, then the class number of the field~$K = \Q[x]/(h)$ is equal to the number of abelian varieties in~$\cI_h$ whose endomorphism ring is maximal; see \cite[Thm.~6.2]{Wat69}. In particular, if~$K$ has class number 1, then any abelian variety~$A$ in~$\cI_h$ whose endomorphism ring is~$\End(A) = \cO_K$ must be self-dual. It is easy to find isogeny classes satisfying this property which do not contain any principally polarizable abelian varieties by using \cite[Thm.~1.3]{Howe95}, for example. See the LMFDB \cite{LMFDB} isogeny class 
  \href{http://www.lmfdb.org/Variety/Abelian/Fq/2/2/ab_ab}{2.2.ab\_ab}
  for a concrete ordinary example, or
  \href{http://www.lmfdb.org/Variety/Abelian/Fq/4/2/ad_c_f_an}{4.2.ad\_c\_f\_an}
  for one which is also geometrically simple.
\end{example}

One way to find examples of abelian varieties which are not self-dual is to first use the categorical equivalence in Theorem~\ref{thm:cat_eq} to compute all isomorphism classes,
and then use Theorem~\ref{thm:cat_eq}.\ref{thm:cat_eq:Avee} to determine which classes are self-dual.
Alternatively, one may use the following proposition which only requires inspecting the local properties of orders in the endomorphism algebra.
The latter technique easily finds Example~\ref{ex:same_end_not_sd}.
\begin{prop}\label{prop:not_self_dual}
  Let~$A$ be any abelian variety in~$\cI_h$ satifying {\bf Ord} or {\bf CS},
   let~$S$ be an order in~$K$ such that~$S=\bar{S}$, and let~$\p$ be a prime of~$S$ such that~$\type_\p(S)=2$ and~$\p=\bar{\p}$ .
  If~$S\subseteq \End(A)$ and~$S_\p = \End(A)_\p$, then~$A$ is not self-dual.
  In particular, such an~$A$ is not principally polarizable and cannot be a Jacobian.
\end{prop}
\begin{proof}
  This follows from Proposition~\ref{prop:cmtype2notselfdual} and Theorem~\ref{thm:cat_eq}.\ref{thm:cat_eq:Avee}.
\end{proof}
% \begin{remark}
%   In addition to {\bf Ord} and {\bf CS}, there is a third condition which we may consider. The {$p$-rank} of an abelian variety~$A$ over a field of characteristic~$p$ is the rank of the geometric~$p$-torsion subgroup of~$A$.  
%   Recall that~$A$ is ordinary if the {$p$-rank} equals~$\dim(A)$, and almost ordinary if it is~$\dim(A) -1$. We write {\bf AOrd} for the condition that~$A$ is an almost ordinary abelian variety over a finite field of odd characteristic.
%   By the combination of \cite[Thm.~4.5]{OswalShankarEarlyView} and \cite[Thm~2.12]{BergKarMar21}, there is a categorical description of isogeny classes satisfying {\bf AOrd} in terms of fractional ideals.
%   As a consequence, Thereom~\ref{thm:cat_eq}.\ref{thm:cat_eq:Avee} holds for abelian varieties satisfying {\bf AOrd}.
%   Therefore, Proposition~\ref{prop:not_self_dual} is also true for abelian varieties satisfying {\bf AOrd}.
  
%   The functors in the {\bf AOrd} and {\bf Ord} cases are constructed similarly.
%   However, our proof of Theorem \ref{thm:cat_eq}.\ref{thm:cat_eq:grp_pts_A}-\ref{thm:cat_eq:grp_pts_Avee} in the {\bf Ord} case  does not apply to the {\bf AOrd} case because it 
%   builds upon results of Howe (\cite[Lem.~4.13,Prop.~4.14]{Howe95}) which require the abelian varieties to be ordinary.
% \end{remark}

\begin{example}[Same endomorphism ring but not self-dual]\
  \label{ex:same_end_not_sd}\
  We go back to the isogeny class~$\cI_h$ from Example~\ref{ex:diff_grps_same_end}.
  One computes that~$R$ has a unique minimal overorder~$T$, and~$[T:R]=2$. Such an order is then necessarily stable under complex conjugation, that is,~$T=\bar{T}$. Also,~$T$ has a unique prime~$\mathfrak{q}$ above~$2$ which also then satisfies~$\mathfrak{q}=\bar{\mathfrak{q}}$. 
  This prime is the unique non-invertible prime of~$T$ and we have~$\type_{\mathfrak{q}}(T)=2$.
  Recall that abelian varieties in~$\cI_h$ with endomorphism ring~$T$ exist by Theorem~\ref{thm:cat_eq}.
  By Proposition~\ref{prop:A_cmtype2_gp}, every abelian variety with endomorphism ring~$T$ has group of rational points isomorphic to~$T/(1-\pi)T$.
  On the other hand, by Proposition~\ref{prop:not_self_dual}, we see that~$A\not\cong A^\vee$ for every abelian variety~$A$~with endomorphism ring~$T$.
\end{example}

We observe that the non-self-dual abelian variety~$A$ found in Example~\ref{ex:same_end_not_sd} also satisfies~$A(\F_5)\cong A^\vee(\F_5)$, thereby exhibiting \ref{it:gp_dual_same}$\nRightarrow$\ref{it:sd} in Theorem~\ref{thm:jac_ppav_implications}. This is also exhibited in the following example, which additionally proves \ref{it:gp_dual_same}$\nRightarrow$\ref{it:end_stable} in the same theorem.

\begin{example}[Isomorphic groups but different endomorphism ring]\
  \label{ex:same_gp_not_same_end}\
  Consider the ordinary isogeny class~$\cI_h$ of abelian surfaces defined over~$\F_3$ determined by the polynomial
  \[ h = x^4 - x^3 + 4x^2 - 3x + 9=(x^2 - 2x + 3)(x^2 + x + 3). \]
  The order~$R=\Z[\pi,\bar \pi]$ has index~$[\cO_K:R]=9$ in the maximal order~$\cO_K$ of~$K$. 
  Since~$h(1)=10$ is coprime with the conductor~$(R:\cO_K)$, we deduce that~$\cI_h$ is cyclic, by Theorem~\ref{thm:coprime_iff_cyclic}. 
  Moreover, since~$10$ is a squarefree integer, we get that~$\cI_h$ is also trivially rich; see Theorem~\ref{thm:rich_condition}.

  We observe that~$R$ has exactly two primes above the singular rational prime~$3$. These two primes are complex conjugates to each other, and we denote them by~$\p$ and~$\bar \p$. 
  There are only two orders between~$R$ and~$\cO_K$, both with index~$3$. These can be realized as the multiplicator rings~$S=(\p:\p)$ and~$\bar S=(\bar \p:\bar \p)$.
  By Theorem~\ref{thm:cat_eq}, we conclude that there is an abelian variety~$A$ in~$\cI_h$ such that~$\End(A)=S$ and~$\End(A^\vee)=\bar{S}$ are not equal, but~$A(\F_3)\cong A^\vee(\F_3) \cong \Z/10\Z$. 
\end{example}

\bibliographystyle{amsalpha}
\renewcommand{\bibname}{References} % changes the header from Bibliography to References
\bibliography{references} % adjust this to fit your BibTex file

\end{document}